\documentclass[twoside]{article}  

\usepackage[margin=2cm,bottom=2.5cm]{geometry} 
\usepackage{amsmath,amsthm,amssymb,latexsym,microtype}
\usepackage[v2,tips]{xy}

\newcommand{\vnten}{\overline\otimes}

\newcommand{\mc}[1]{\mathcal{#1}}
\newcommand{\ip}[2]{{\langle {#1} , {#2} \rangle}}
\newcommand{\op}{{\operatorname{op}}}
\newcommand{\G}{{\mathbb G}}

\theoremstyle{plain}
\newtheorem{proposition}{Proposition}[section]
\newtheorem{theorem}[proposition]{Theorem}
\newtheorem{corollary}[proposition]{Corollary}
\newtheorem{lemma}[proposition]{Lemma}
\theoremstyle{definition}
\newtheorem{definition}[proposition]{Definition}

\begin{document}

\large
\title{Multipliers of locally compact quantum groups via Hilbert C$^*$-modules}
\author{Matthew Daws}
\maketitle

\begin{abstract}
A result of Gilbert shows that every completely bounded multiplier $f$ of the Fourier
algebra $A(G)$ arises from a pair of bounded continuous maps $\alpha,\beta:G
\rightarrow K$, where $K$ is a Hilbert space, and $f(s^{-1}t) = (\beta(t)|\alpha(s))$
for all $s,t\in G$.  We recast this in terms of adjointable operators acting
between certain Hilbert C$^*$-modules, and show that an analogous construction
works for completely bounded left multipliers of a locally compact
quantum group.  We find various ways to deal with right multipliers: one of these
involves looking at the opposite quantum group, and this leads to a proof that the
(unbounded) antipode acts on the space of completely bounded multipliers, in a way which
interacts naturally with our representation result.
The dual of the universal quantum group (in the sense of Kustermans) can be identified with
a subalgebra of the completely bounded multipliers, and we show how this fits into our
framework.  Finally, this motivates a certain way to deal with two-sided multipliers.

(2010) Subject classification: 43A22, 46L08, 46L89 (Primary);
22D15, 22D25, 22D35, 43A30 (Secondary).

\end{abstract}

\section{Introduction}

Let $G$ be a locally compact group $G$, and let $A(G)$ be the Fourier algebra of $G$, the
subalgebra of $C_0(G)$ given by coefficient functionals of the left regular
representation $\lambda$ of $G$ on $L^2(G)$, see \cite{eymard}.  A \emph{multiplier}
of $A(G)$ is a continuous function $f\in C^b(G)$ such that $fa\in A(G)$ for
each $a\in A(G)$.  A multiplier $f$ induces an automatically bounded map
$A(G)\rightarrow A(G)$.  As $A(G)$ is the predual of the group von Neumann
algebra $VN(G)$, it carries a natural operator space structure, and so we
can ask when the map induced by $f$ is completely bounded.  The collection of
such $f$ is the algebra of \emph{completely bounded multipliers} of $A(G)$,
written $M_{cb}A(G)$.  A result of Gilbert (see \cite{BF}, the short proof in \cite{J},
the introduction of \cite{ch}, or the survey \cite{spronk})
shows that $f\in M_{cb}A(G)$ if and only if there is a Hilbert space $K$
and bounded continuous functions $\alpha,\beta:G\rightarrow K$ with
\[ f(s^{-1}t) = \big( \beta(t) \big| \alpha(s) \big) \qquad (s,t\in G), \]
where $(\cdot|\cdot)$ denotes the inner-product on $K$.  (This formula has
$s^{-1}t$ instead of $t^{-1}s$ as considered by Jolissaint in \cite{J};
see Section~\ref{multdual} below for an explanation).

In this paper, we shall propose variations of this result for the convolution
algebra $L^1(\G)$ of a locally compact quantum group $\G$ (see below for definitions).
Clearly the space of continuous functions $G\rightarrow K$ will be important,
and we start with a short discussion of this.  Indeed, consider the
C$^*$-algebra $A=C_0(G)$.  Let $A\otimes K$ be the standard Hilbert C$^*$-module
(see \cite{lance}) which in this case can be identified with $C_0(G,K)$.  Then
the ``multiplier space'' of $A\otimes K$ is identified with $C^b(G,K)$; abstractly,
this is the space $\mc L(A,A\otimes K)$ of adjointable maps from $A$ to $A\otimes K$.
To induce a member of $M_{cb}A(G)$, we need that the pair $(\alpha,\beta)$ is
``invariant'' in the sense that $(\beta(t^{-1})|\alpha(t^{-1}s^{-1})) = f(s)$
for all $s,t\in G$.

In the quantum setting, we replace $C_0(G)$ be a possibly non-commutative C$^*$-algebra,
denoted $C_0(\G)$.  The dual quantum group to $C_0(G)$ is $C^*_r(G)$, and the Fourier
algebra is the predual of the $VN(G)=C^*_r(G)''$.  Thus, by analogy,
we will study completely bounded multipliers of the convolution algebra of the
dual quantum group, denoted $L^1(\hat\G)$.  Indeed, we work firstly by
looking at completely bounded \emph{left} multipliers of $L^1(\hat\G)$.  We restrict
attention to those multipliers which are ``represented'' by some $x\in C^b(\G)$ (so that
under the regular representation $\hat\lambda:L^1(\hat\G) \rightarrow C_0(\G)$, left
multiplication by $x$ induces our left multiplier).  This is automatic for the left part
of two-sided multipliers, see \cite[Section~8.2]{dawsm}.
In this setting, we get a complete analogy of Gilbert's result.  To study right
multipliers, we can either use the unitary antipode, or study the opposite algebra
$L^1(\hat\G)^\op$.  These turn out \emph{not} to be totally equivalent,
and the study of $L^1(\hat\G)^\op$ leads us to study how the (unbounded, in general)
antipode of $\G$ acts on the space of multipliers.  A corollary is that two-sided
multipliers are invariant under the action of the antipode.  Furthermore, this now
puts us in a position to use the representation result of Junge, Neufang and Ruan
proved in \cite{jnr}, which implies that in fact every completely bounded left
multiplier is represented in our sense.  By taking a different
perspective on the space $\mc L(A,A\otimes K)$, we are lead to consider ideas very
close to those studied by Vaes and Van Daele in \cite{vvd}.

In the final part of the paper, we look at the universal quantum group (in the
sense of Kustermans, \cite{kus1}) of $\hat\G$.  This always induces completely
bounded multipliers of $L^1(\hat\G)$, and we show how this fits into our framework.
Motivated by this construction, we end by giving one, reasonably symmetric, way to
deal with two-sided multipliers.

We follow \cite{lance} for the theory of Hilbert C$^*$-modules.  In particular,
all our inner-products will be linear in the \emph{second variable}, and we consider
\emph{right} (Hilbert C$^*$-)modules.  We similarly often let scalars act on the
right of a vector space.

\medskip

\noindent\textbf{Acknowledgements:}
We thank Martin Lindsay for suggesting the
idea of viewing $\mc L(A,A\otimes K)$ as a ``corner'' or ``slice'' of
$\mc L(A\otimes K)$; this both simplifies proofs in Section~\ref{cstarmods}
and also provides motivation for our treatment of two-sided multipliers.
While visiting Leeds on the EPSRC grant EP/I002316/1, Zhong-Jin Ruan and Matthias
Neufang pointed the author in the direction of \cite{jnr}, and Nico
Spronk suggested the comment about $\operatorname{wap}(G)$ in Section~\ref{sec:coord}.
Finally, the anonymous referee provided many helpful comments which have substantially
improved the paper.

\section{Locally compact quantum groups and multipliers}\label{sec:lcqg}

In this section, we sketch (rather briefly) the theory of locally compact quantum
groups; our main aim is to fix notation.  For details on the von Neumann algebraic
side of the theory, see \cite{kus2}, and for the C$^*$-algebraic side, see
\cite{kus} and \cite{mas}.  The survey \cite{kus5}, and Vaes's PhD thesis \cite{vaes},
are gentle, well-motivated introductions.

A \emph{locally compact quantum group} is a von Neumann algebra $M$ together with a
coproduct $\Delta:M\rightarrow M\vnten M$.  This is a unital normal $*$-homomorphism
with $(\Delta\otimes\iota)\Delta = (\iota\otimes\Delta)\Delta$.  Furthermore, we assume
the existence of left and right invariant weights on $M$.  The coproduct $\Delta$ turns
the predual $M_*$ into a completely contractive Banach algebra.

Associated to $(M,\Delta)$ is a \emph{reduced} C$^*$-algebraic quantum group $(A,\Delta)$.
Here $A$ is a C$^*$-subalgebra of $M$, and $\Delta:A\rightarrow M(A\otimes A)$, the
multiplier algebra of $A\otimes A$, the minimal C$^*$-algebra tensor product (which is
the only tensor product of C$^*$-algebras which we shall consider).  Here we
identify $M(A\otimes A)$ with a subalgebra of $M\vnten M$.  The dual space $A^*$ becomes
a completely contractive Banach algebra which contains $M_*$ as a closed ideal.

We use the left invariant weight to build a Hilbert space $H$; then $M$ is in standard
position on $H$.  There is a privileged unitary operator $W$ on $H\otimes H$
(the Hilbert space tensor product of $H$ with itself) with
$\Delta(x) = W^*(1\otimes x)W$ for $x\in M$.
Then $W$ is a \emph{multiplicative unitary}, and $W\in M(A \otimes \mc B_0(H))$,
where $\mc B_0(H)$ is the algebra of compact operators on $H$.
Define $\lambda:M_*\rightarrow\mc B(H)$
by $\lambda(\omega) = (\omega\otimes\iota)(W)$.  Let the closure of $\lambda(M_*)$ be
$\hat A$, which is a C$^*$-algebra.  Let $\hat M$ be the $\sigma$-weak closure,
which is a von Neumann algebra.  We may define a coproduct $\hat\Delta$ on $\hat M$
by $\hat\Delta(x) = \hat W^*(1\otimes x)\hat W$, where $\hat W = \sigma W^*\sigma$,
where $\sigma$ is the flip map on $H\otimes H$.  It is possible to construct left
and right invariant weights on $\hat M$, turning this into a locally compact quantum
group, whose $C^*$-algebraic counterpart is $\hat A$.  We have the biduality theorem,
that $\hat{\hat M} = M$ canonically.

As is becoming common, we write $\mathbb G$ for an abstract object, to be thought of
as a locally compact quantum group, and
we write $L^1(\G), L^\infty(\G), C_0(\G), C^b(\G)$ and $M(\G)$ for, respectively,
$M_*, M, A, M(A)$ and $A^*$.  We shall then write $\hat\G$ for the abstract object
corresponding to the dual quantum group, so that $\hat M$ is denoted by $L^\infty(\hat\G)$,
and so forth.  We shall use the hat notation to signify that an object should be thought
of as corresponding to $\hat\G$.  For example, for $\xi,\eta\in L^2(\G)$, we have the
vector functional $\omega_{\xi,\eta}:\mc B(L^2(\G)) \rightarrow \mathbb C; x \mapsto
(\xi|x\eta)$, and then the restriction of this to $L^\infty(\hat\G)$ is denoted by
$\hat\omega_{\xi,\eta}\in L^1(\hat\G)$.

Locally compact quantum groups generalise Kac algebras (see \cite{kac} and
\cite[Page~7]{vv}).  However, unlike for a Kac algebra, $L^1(\G)$ need not be $*$-algebra,
as the antipode $S$ is in general unbounded.  However, $L^1(\G)$ contains a
dense $*$-subalgebra $L^1_\sharp(\G)$.  This is the space of functionals $\omega\in L^1(\G)$
such that there exists $\sigma\in L^1(\G)$ with $\ip{x}{\sigma} = \ip{S(x)}{\omega^*}$
for $x\in D(S)$, the domain of $S$.  Here $\omega^*$ is the functional given by
$\ip{y}{\omega^*} = \overline{\ip{y^*}{\omega}}$ for $y\in L^\infty(\G)$.  We write
$\sigma = \omega^\sharp$ in this case, and then $\lambda(\omega^\sharp)
= \lambda(\omega)^*$.  See \cite[Section~3]{kus1} or \cite[Section~2]{kus2} for further
details.

As we are working with right multipliers, to avoid a notational clash,
we shall write $\kappa$ (and not $R$) for the
unitary antipode on $L^\infty(\G)$.  This is a normal anti-$*$-homomorphism with
$(\kappa\otimes\kappa)\sigma\Delta = \Delta\kappa$.  Thus the pre-adjoint $\kappa_*$
is an anti-homomorphism of $L^1(\G)$.  Furthermore, $\kappa$ is spatially implemented,
as $\kappa(x) = \hat J x^* \hat J$ for $x\in L^\infty(\G)$, where $\hat J$ is the
modular conjugation for (the left weight of) $\hat\G$.  The unitary antipodes
interact well with duality, in that $\kappa \hat\lambda = \hat\lambda \hat\kappa_*$.

There is a one-parameter group of automorphisms $(\tau_t)$ of $C_0(\G)$ which
links $S$ and $\kappa$, by $S = R \tau_{-i/2}$.  Then $R$ commutes with $(\tau_t)$,
so also $S = \tau_{-i/2} R$, and we see that $D(S) = D(\tau_{-i/2})$.
The group $(\tau_t)$ extends to a group of
automorphisms, continuous for the $\sigma$-strong$^*$ topology, of $L^\infty(\G)$.

As we are looking at the left regular representation, it is natural that things work
best for us when looking at left multipliers.  We shall later deal with right
multipliers: these can be converted to left multipliers
by looking at the \emph{opposite algebra}.
At the quantum group level, we define $\hat\G^\op$ to be the \emph{opposite} quantum group
to $\hat\G$, see \cite[Section~4]{kus2}.
That is, $L^\infty(\hat\G^\op) = L^\infty(\hat\G)$, but the multiplication in
$L^1(\hat\G^\op)$ is reversed from that in $L^1(\hat\G)$.  This is equivalent to
defining the comultiplication on $L^\infty(\hat\G^\op)$ to be $\sigma\hat\Delta$.

Then we have that $L^\infty((\hat\G^\op)\hat{}) = L^\infty(\G)'$, the commutant
of $L^\infty(\G)$ in $\mc B(L^2(\G))$.  Let the resulting locally compact quantum
group be denoted by $\G'$.  The natural coproduct $\Delta'$ is defined as follows,
where $J$ is the modular conjugation on $L^\infty(\G)$,
\[ \Delta'(x) = (J\otimes J) \Delta(JxJ) (J\otimes J) \qquad
(x\in L^\infty(\G') = L^\infty(\G)'), \]
The associated multiplicative unitary is $W' = (J\otimes J)W(J\otimes J)$.
Then $C_0(\G')$ is the norm closure of
$\{ (\iota\otimes\omega)(W') : \omega\in\mc B(L^2(\G))_* \}$, which is easily seen
to be $JC_0(\G)J$.  Consider the unitary map $\hat JJ$, and for $x\in C_0(\G')$
define $\Phi(x) = \hat J JxJ \hat J = \kappa(JxJ)^* \in C_0(\G)$, so that $\Phi$ is a
C$^*$-isomorphism of $C_0(\G')$ to $C_0(\G)$.  We then get the
\emph{right regular representation} $\hat\rho:L^1(\hat\G)
\rightarrow C_0(\G'); \hat\omega\mapsto \Phi(\hat\lambda(\hat\omega))$.

\subsection{Multipliers and duality}\label{multdual}

For a Banach algebra $\mc A$, a \emph{(two-sided) multiplier} (also called a
\emph{(double) centraliser}) is a pair of maps $L,R:\mc A\rightarrow\mc A$ such that
$aL(b) = R(a)b$.  We write $(L,R)\in M(\mc A)$, and then
$M(\mc A)$ becomes an algebra for the product $(L,R)(L',R') = (LL',R'R)$.
We shall always suppose that $\mc A$ is \emph{faithful}, that is,
if $bac=0$ for all $b,c\in\mc A$, then $a=0$.  In this case, we can show
that $L(ab)=L(a)b$ and $R(ab)=aR(b)$.  A closed graph argument will show
that $L$ and $R$ are automatically bounded.  There is a natural map
(injective, as $\mc A$ is faithful) of $\mc A$ into $M(\mc A)$ given by
$a\mapsto (L_a,R_a)$ where $L_a(b) = ab$ and $R_a(b) = ba$ for $b\in\mc A$.
For further details, see \cite{dales}, \cite[Section~1.2]{palmer} or \cite{dawsm}.

When $\mc A$ is a completely contractive Banach algebra, we can restrict attention
to those $(L,R)\in M(\mc A)$ such that $L$ and $R$ are completely bounded.  We write
$M_{cb}(\mc A)$ for the algebra of \emph{completely bounded multipliers}.  If $\mc A$
has a bounded approximate identity, then $M(\mc A)=M_{cb}(\mc A)$ with equivalent
norms, see \cite[Proposition~3.1]{KR} or \cite[Theorem~6.2]{dawsm}. Otherwise, there
appears to be no general relationship between $M(\mc A)$ and $M_{cb}(\mc A)$.

We shall also work with \emph{left multipliers}, that is, bounded maps
$L:\mc A\rightarrow\mc A$ with $L(ab)=L(a)b$ for $a,b\in\mc A$.  We write
$L\in M^l(\mc A)$.  Similarly, we define the \emph{right multipliers} $M^r(\mc A)$,
and the analogous completely bounded versions, $M^l_{cb}(\mc A)$ and $M^r_{cb}(\mc A)$.

\begin{definition}
Let $\G$ be a locally compact quantum group.  A multiplier $L\in M^l(L^1(\hat\G))$
is \emph{represented} if there exists $a\in C^b(\G)$ such that $\hat\lambda(L(\hat\omega))
= a \hat\lambda(\hat\omega)$ for each $\hat\omega\in L^1(\hat\G)$.  Similarly,
$R\in M^r(L^1(\hat\G))$ is \emph{represented} if there exists $a\in C^b(\G)$ such that $\hat\lambda(R(\hat\omega)) = \hat\lambda(\hat\omega)a$ for each $\hat\omega\in L^1(\hat\G)$.
\end{definition}

Building on work of Kraus and Ruan in \cite{KR}, we showed in \cite[Theorem~8.9]{dawsm}
that a two-sided multiplier $(L,R)\in M_{cb}(L^1(\hat\G))$ is represented by some
$a\in C^b(\G)$; that is, $a\hat\lambda(\omega) = \hat\lambda(L(\omega))$ and
$\hat\lambda(\omega)a = \hat\lambda(R(\omega))$ for each $\hat\omega\in L^1(\hat\G)$
(compare with Proposition~\ref{prop:four} below).
The resulting map $\hat\Lambda:M_{cb}(L^1(\hat\G))\rightarrow C^b(\G)$ is
a completely contractive algebra homomorphism.
We remark that we don't know if $\hat\Lambda$ can be extended (even just as an
algebra homomorphism) to $M(L^1(\hat\G))$.

To illustrate this, let
$G$ be a locally compact group, and form the commutative quantum group $L^\infty(G)$.
Here the coproduct is given by $\Delta(F)(s,t) = F(st)$ for $F\in L^\infty(G)$ and
$s,t\in G$.  The left and right invariant weights are given by integrating against the
left and right Haar measures, respectively.  Then the dual quantum group is $VN(G)$,
which has predual $A(G)$, the Fourier algebra.  The associated Hilbert space is simply
$L^2(G)$, and as $VN(G)$ is in standard position, every normal functional $\omega\in A(G)$
is of the form $\omega_{\xi,\eta}$, where $\ip{x}{\omega_{\xi,\eta}} = (\xi|x(\eta))$
for $x\in VN(G)$ and $\xi,\eta\in L^2(G)$.  The multiplicative unitary is given
by $W\xi(s,t) = \xi(s,s^{-1}t)$ for $\xi\in L^2(G\times G), s,t\in G$.  Let
$\lambda:G\rightarrow\mc B(L^2(G))$ be the left regular representation, where
\[ \lambda(s):\xi\mapsto \eta, \quad \eta(t) = \xi(s^{-1}t)
\qquad (\xi\in L^2(G), s,t\in G). \]
This does integrate to give the expected map $\lambda:L^1(G)\rightarrow\mc B(L^2(G))$.
Then $\hat\lambda:A(G) \rightarrow C_0(G)$, and we can check that
\[ \hat\lambda(\omega)(s) = \ip{\lambda(s^{-1})}{\omega} \qquad (s\in G, \omega\in A(G)). \]
Thus $\hat\lambda$ gives the map considered by Takesaki in
\cite[Chapter~VII, Section~3]{tak2}, and \emph{not} the map considered by Eymard in
\cite{eymard} (where $s^{-1}$ is replaced by $s$).  This also explains why our formulas
in the introduction were different to those considered Jolissaint in \cite{J}, as the
embedding $\hat\Lambda:M_{cb}A(G) \rightarrow C^b(G)$ is consequently also different to
that usually considered.

A representation result for completely bounded multipliers was shown by Junge, Neufang
and Ruan in \cite{jnr}.  The principle result of that paper is \cite[Theorem~4.5]{jnr},
which shows a completely isometric identification between $M_{cb}^r(L^1(\hat\G))$ and
$\mc{CB}^{\sigma,L^\infty(\hat\G)}_{L^\infty(\G)}(\mc B(L^2(\G)))$.  This latter space is
the algebra of weak$^*$-continuous, completely bounded maps $\mc B(L^2(\G)) \rightarrow
\mc B(L^2(\G))$ which are $L^\infty(\G)$-bimodule maps, and which map $L^\infty(\hat\G)$
into itself.  This space can be also be studied by using the (extended) Haagerup tensor
product, and it is possible to view our constructions using a viewpoint similar to
the Haagerup tensor product-- this is explored in Section~\ref{sec:coord} below;
in some sense, our results are C$^*$-algebraic counterparts to the von Neumann algebra
approach of \cite{jnr}.  Indeed, \cite{jnr} was preceded by work of Neufang, Ruan and
Spronk in \cite{nrs} where the $L^1(G)$ and $A(G)$ cases are worked out.
Links with the Haagerup tensor product, and Gilbert's theorem, are explicitly used
in \cite[Section~4]{nrs}.

For us, the importance of \cite{jnr} is the following result (recall the discussion in
the previous section about the right regular representation $\hat\rho$).

\begin{theorem}[{\cite[Corollary~4.4]{jnr}}]\label{thm:ten}
Let $R\in M^r_{cb}(L^1(\hat\G))$.  There exists $x\in L^\infty(\G')$ such that
$\hat\rho(\hat\omega)x = \hat\rho(R(\hat\omega))$ for all $\hat\omega\in L^1(\hat\G)$.
\end{theorem}

Actually, the full power of the representation result of \cite{jnr} is not needed
to show this-- see Section~\ref{sec:coord} below where
a very brief sketch of the proof is given.  However, undoubtedly the proof of this
result is more ingenuous than the two-sided multiplier case.

Using the unitary antipode, it's easy to transfer this result to completely
bounded left multipliers.  Notice that we only
get $x\in L^\infty(\G')$, not $C^b(\G')$, which is slightly weaker than the requirement
for $R$ to be \emph{represented} in our sense.  However, we are able to boot-strap this
result and show that actually $x$ is in $C^b(\G')$ (see
Theorem~\ref{thm:three} and Proposition~\ref{prop:ten} below).

The following results extract a little bit more information than we found in
\cite{dawsm}; they are also similar to, for example, \cite[Theorem~4.10]{jnr}.

\begin{proposition}\label{prop:one}
Let $R$ be a normal completely bounded map $L^\infty(\hat\G)\rightarrow\mc B(L^2(\G))$,
and let $a\in\mc B(L^2(\G))$ be such that $(R\otimes\iota)(\hat W) = \hat W(1\otimes a)$.
Then $R$ maps into $L^\infty(\hat\G)$, and the pre-adjoint $R_*$ is a right multiplier
of $L^1(\hat\G)$ with $\hat\lambda(R_*(\hat\omega)) = \hat\lambda(\hat\omega)a$ for
$\hat\omega\in L^1(\G)$.  Furthermore, $a\in L^\infty(\G)$.
\end{proposition}
\begin{proof}
Let $T(L^2(\G))$ be the trace-class operators on $L^2(\G)$, and let
$q:T(L^2(\G))\rightarrow L^1(\hat\G)$ be the natural quotient map, which is
actually a complete quotient map, see \cite[Section~4.2]{ER}.  Let $\xi_0,\eta_0,\xi,
\eta\in L^2(\G)$, so that
\begin{align*} \big( \xi \big| \hat\lambda(\omega_{\xi_0,\eta_0}) a \eta \big)
&= \big( \xi \big| (\omega_{\xi_0,\eta_0}\otimes\iota)(\hat W) a \eta \big)
= \big( \xi_0\otimes\xi \big| \hat W(\eta_0 \otimes a\eta) \big) \\
&= \big( \xi_0\otimes\xi \big| (R\otimes\iota)(\hat W)(\eta_0\otimes\eta) \big)
= \ip{\hat W}{R_*(\omega_{\xi_0,\eta_0}) \otimes \omega_{\xi,\eta}}
= \big( \xi \big| \hat\lambda(R_*(\omega_{\xi_0,\eta_0})) \eta \big)
\end{align*}
Thus $\hat\lambda(q(\omega))a = \hat\lambda(R_*(\omega))$ for $\omega\in T(L^2(\G))$.

In particular, $\hat\lambda(\hat\omega)a \in \hat\lambda(L^1(\hat\G))$ for each
$\hat\omega\in L^1(\hat\G)$.  As $\hat\lambda$ is injective, there exists some function
$r:L^1(\hat\G)\rightarrow L^1(\hat\G)$ with $\hat\lambda(\hat\omega)a = 
\hat\lambda(r(\hat\omega))$ for $\hat\omega\in L^1(\hat\G)$.  Using again
that $\hat\lambda$ is an injective homomorphism, it is easy to check that $r$ is linear
and a right multiplier (but maybe not bounded).  However, we then see that
\[ \hat\lambda\big( q(\omega) \big) a = \hat\lambda\big( r(q(\omega)) \big)
= \hat\lambda\big(R_*(\omega)\big) \qquad (\omega\in T(L^2(\G))). \]
So $R_* = rq$ and hence $R_*$ drops to a completely bounded map
$L^1(\hat\G)\rightarrow L^1(\hat\G)$, and then $r=R_*$, as required.

Finally, let $x\in L^\infty(\G)'$, so for $\hat\omega\in L^1(\hat\G)$,
we have that $\hat\lambda(\hat\omega) ax = \hat\lambda(R_*(\hat\omega)) x
= x \hat\lambda(R_*(\hat\omega)) = x \hat\lambda(\hat\omega) a
= \hat\lambda(\hat\omega) xa$.  This is enough to imply that
$ax=xa$ (compare with the proof of \cite[Proposition~8.8]{dawsm}) and so
$a\in L^\infty(\G)'' = L^\infty(\G)$.
\end{proof}

It is easily checked that, similarly, when $L\in\mc{CB}(L^\infty(\hat\G),\mc B(L^2(\G)))$
is normal with there existing $a\in\mc B(L^2(\G))$ with $(L\otimes\iota)(\hat W)
= (1\otimes a)\hat W$, then $L\in\mc{CB}(L^\infty(\hat\G))$, and the pre-adjoint
$L_*$ is a left multiplier of $L^1(\G)$ with $\hat\lambda(L_*(\hat\omega))
= a\hat\lambda(\hat\omega)$ for $\hat\omega\in L^1(\G)$.

Similarly, if $(L,R)$ is a pair of maps, both associated to the same
$a\in\mc B(L^2(\G))$, then the pre-adjoints form a multiplier
$(L_*,R_*) \in M_{cb}(L^1(\hat\G))$.  As $\hat\lambda(L^1(\hat\G))$ is
dense in $C_0(\G)$, it follows that automatically $a\in C^b(\G)$.
We shall later prove that this is true for one-sided multipliers as well,
see Theorem~\ref{thm:three} and Proposition~\ref{prop:ten} below.

\begin{proposition}\label{prop:four}
Let $L_*\in M^l_{cb}(L^1(\hat\G))$ and $a\in L^\infty(\G)$ be such that
$a \hat\lambda(\hat\omega) = \hat\lambda(L_*(\hat\omega))$ for $\hat\omega\in L^1(\hat\G)$.
Setting $L=L_*^*$, we have that $(L\otimes\iota)(\hat W) = (1\otimes a)\hat W$.
Similarly, if $R$ is the adjoint of a completely bounded right
multiplier associated to $a$, then $(R\otimes\iota)(\hat W) = \hat W(1\otimes a)$.
If $L_*$ and $R_*$ are both associated to the same $a\in L^\infty(\G)$, then
$(L_*,R_*)\in M_{cb}(L^1(\hat\G))$ and $a\in C^b(\G)$.
\end{proposition}
\begin{proof}
We simply reverse some previous calculations, where, for variety, we work with left
multipliers.  For $\xi_0,\eta_0,\xi,\eta\in L^2(\G)$, we have
\begin{align*} \big( \xi \big| a \hat\lambda(\hat\omega_{\xi_0,\eta_0}) \eta \big)
&= \big( \xi \big| \hat\lambda(L_*(\hat\omega_{\xi_0,\eta_0})) \eta \big)
= \ip{\hat W}{L_*(\hat\omega_{\xi_0,\eta_0}) \otimes \omega_{\xi,\eta}}
= \ip{(L\otimes\iota)(\hat W)}{\hat\omega_{\xi_0,\eta_0} \otimes \omega_{\xi,\eta}} \\
&= \big( a^*\xi \big| (\hat\omega_{\xi_0,\eta_0}\otimes\iota)(\hat W) \eta \big)
= \ip{\hat W}{\hat\omega_{\xi_0,\eta_0} \otimes \omega_{a^*\xi,\eta}}
= \ip{(1\otimes a)\hat W}{\hat\omega_{\xi_0,\eta_0} \otimes \omega_{\xi,\eta}}.
\end{align*}
Thus $(L\otimes\iota)(\hat W) = (1\otimes a)\hat W$.  A similar calculation holds for
right multipliers.

If $L_*$ and $R_*$ are associated to the same $a$, then for $\hat\omega,
\hat\sigma\in L^1(\G)$,
\[ \hat\lambda\big( \hat\omega L_*(\hat\sigma) \big)
= \hat\lambda(\hat\omega) a \hat\lambda(\hat\sigma)
= \hat\lambda\big( R_*(\hat\omega) \hat\sigma \big), \]
using that $\hat\lambda$ is a homomorphism.  As $\hat\lambda$ injects, it follows
that $(L_*,R_*)$ is a multiplier, which is completely bounded by assumption.
That now $a\in C^b(\G)$ follows from the comment above.
\end{proof}

\section{Hilbert C$^*$-modules}\label{cstarmods}

We shall use the basic theory of Hilbert C$^*$-modules, following \cite{lance}, for
example.  Let us develop a little of this theory.  Given a $C^*$-algebra $A$ and a
Hilbert space $K$, we let $A\odot K$ be the algebraic tensor product of $A$ with $K$,
turned into a right $A$-module in the obvious way, and given the $A$-valued inner-product
$(a\otimes\xi|b\otimes\eta) = a^*b (\xi|\eta)$.  Let $A\otimes K$ be the completion.

Let $E$ and $F$ be Hilbert C$^*$-modules over $A$.  We write $\mc K(E,F)$ for the
``compact'' operators from $E$ to $F$, the closure of the linear span of maps
$\theta_{x,y}$.  Here $x\in F, y\in E$ and we have $\theta_{x,y}(z) = x (y|z)$
for $z\in E$.  Let $\mc L(E,F)$ be the space of all adjointable operators from
$E$ to $F$.  Recall that the unit ball of $\mc K(E,F)$ is strictly dense in the
unit ball of $\mc L(E,F)$.  When $E=F$, we can identify $\mc L(E)$ with the
multiplier algebra $M(\mc K(E))$.  Indeed, $\mc K(E)$ is an essential ideal in
$\mc L(E)$, so we have an inclusion $\mc L(E) \rightarrow M(\mc K(E))$, which
is actually surjective.  When $E=F=A$, we have $\mc K(A)=A$ and $\mc L(A)$ is
identified with the multiplier algebra $M(A)$.

We identify $\mc K(A\otimes K)$ with $A \otimes \mc B_0(K)$.  The isomorphism
sends $\theta_{a\otimes\xi,b\otimes\eta}$ to $ab^* \otimes \theta_{\xi,\eta}$.
Here $\theta_{\xi,\eta}\in\mc B_0(K)$ is the finite-rank map $\phi\mapsto
\xi (\eta|\phi)$.  That this extends by continuity is a little subtle; see \cite{lance}.
Notice that if $P\in\mc B(K)$, then $\iota\otimes P\in\mc L(A\otimes K)$.

More generally, let $E$ and $F$ be Hilbert C$^*$-modules over $A$ and $B$,
respectively.  We let $E\otimes F$ be the exterior tensor product, which is
a Hilbert C$^*$-module over $A\otimes B$, with the inner-product
\[ (x\otimes y| w\otimes z) = (x|w) \otimes (y|z). \]
We then have an embedding $\mc L(E) \otimes \mc L(F) \rightarrow \mc L(E\otimes F)$,
and more generally, an embedding of $\mc L(E_1,E_2) \otimes \mc L(F_1,F_2)$ into
$\mc L(E_1\otimes F_1, E_2\otimes F_2)$.

As mentioned in the introduction, for a locally compact space $G$, we may identify
$C_0(G)\otimes K$ with $C_0(G,K)$, the continuous functions from $G$ to $K$ which
vanish at infinity.  Given $\alpha\in C^b(G,K)$, a bounded continuous function from
$G$ to $K$, we define $\mc T\in\mc L(C_0(G),C_0(G)\otimes K)$ by
\[ \mc T(a) = \big( a(s) \alpha(s) \big)_{s\in G} \qquad (a\in C_0(G)). \]
A calculation shows that $\mc T$ is indeed adjointable: if $x\in C_0(G,K)$ then
$\mc T^*(x)(s) = (\alpha(s)|x(s))$ for $s\in G$.  Conversely, it is not too hard to show
that any member of $\mc L(C_0(G),C_0(G)\otimes K)$ arises in this way.

This hence motivates the study of $\mc L(A,A\otimes K)$ for an arbitrary C$^*$-algebra
$A$.  Fix a unit vector $\xi_0\in K$, and regard $K$ as the ``row space''
$\mc L(K, \mathbb C)$, where $K$ is a module over $\mathbb C$.  So $\xi_0$ is identified
with the map $\eta\mapsto(\xi_0|\eta)$.  This is adjointable, with adjoint
$\xi_0^*:\mathbb C\rightarrow K; t\mapsto t\xi_0$.  Let $\iota:A\rightarrow A$ be the
identity, so, as above, we can form the tensor product $\iota\otimes\xi_0 \in
\mc L(A\otimes K,A\otimes\mathbb C) = \mc L(A\otimes K,A)$.  This is simply the map
$a\otimes\eta \mapsto a (\xi_0|\eta)$, and the adjoint is $(\iota\otimes\xi_0)^* =
\iota\otimes\xi_0^*:a \mapsto a\otimes\xi_0$.  It is actually not particularly hard to
show by direct calculation that these maps are contractive and are mutual adjoints.

Then we have an embedding and a quotient map, both
of which are adjointable, and hence $A$-module maps:
\begin{align*} \mc L(A,A\otimes K) &\rightarrow \mc L(A\otimes K)\cong M(A\otimes\mc B_0(K));
\quad \alpha \mapsto \alpha(\iota\otimes\xi_0), \\
\mc L(A\otimes K) &\rightarrow \mc L(A,A\otimes K); \quad
\mc T \mapsto \mc T(\iota\otimes\xi_0)^*. \end{align*}
Hence we can identify $\mc L(A,A\otimes K)$ as a complemented submodule of
$\mc L(A\otimes K)$.  This follows, as $(\iota\otimes\xi_0)(\iota\otimes\xi_0)^*$ is
the identity on $A$, and so the map $\mc T\mapsto \mc T(\iota\otimes\xi_0)^*
(\iota\otimes\xi_0)$ is a projection from $\mc L(A\otimes K)$ onto the image
of $\mc L(A,A\otimes K)$.

We shall use the notation that $\mc T\in\mc L(A\otimes K)$ is identified with
$T\in M(A\otimes\mc B_0(K))$.  Suppose that $A$ is faithfully and non-degenerately
represented on $H$.
Then we can identify $M(A\otimes\mc B_0(K))$ with a subalgebra of $\mc B(H\otimes K)$,
and we shall continue to write $T$ for the resulting operator in $\mc B(H\otimes K)$.
Similarly, we identify $M(A)$ with $\{ T\in\mc B(H) : Ta,aT\in A \ (a\in A) \}$.

It will be useful to define some auxiliary maps.  For $\xi\in H$, define
$e_\xi: A\otimes K\rightarrow H\otimes K$ by $e_\xi(a\otimes\eta) = a(\xi)\otimes\eta$,
and linearity and continuity.
This makes sense, as given $\tau=\sum_n a_n\otimes\eta_n\in A\otimes K$, we have that
\[ \|e_\xi(\tau)\|^2 = \sum_{n,m} (a_n(\xi)|a_m(\xi)) (\eta_n|\eta_m)
= \Big( \xi \Big| \sum_{n,m} a_n^*a_m (\eta_n|\eta_m) \xi\Big)
= \big( \xi \big| (\tau|\tau) \xi \big)
\leq \|\xi\|^2 \|\tau\|^2. \]
Thus $e_\xi$ is bounded, with $\|e_\xi\| \leq \|\xi\|$.  Notice that this calculation
also shows that
\[ \big( e_\xi(\tau) \big| e_\eta(\sigma) \big) = \big( \xi \big| (\tau|\sigma) \eta \big)
\qquad (\tau,\sigma\in A\otimes K, \xi,\eta\in H), \]
where here $(\tau|\sigma)\in A \subseteq \mc B(H)$.

The next two propositions show a tight connection between these ideas.  In the
following, we could have \emph{defined} $\tilde\alpha$ using (iii).  Notice that
as $A$ has a bounded approximate identity and is non-degenerately represented
on $H$, it follows that $H = \{ a(\xi) : a\in A, \xi\in H \}$; this uses the
Cohen Factorisation Theorem (compare \cite[Corollary~2.9.25]{dales} or
\cite[Theorem~A.1]{mas}).  However, we would still have to prove that $\tilde\alpha$
were well-defined.

\begin{proposition}\label{prop:six}
Let $A$ be a C$^*$-algebra faithfully and non-degenerately
represented on $H$, and let $K$ be a Hilbert
space.  Let $\alpha\in\mc L(A,A\otimes K)$ and $\mc T\in\mc L(A\otimes K)$ be related
by $\alpha = \mc T(\iota\otimes\xi_0)^*$, where $\xi_0\in K$ is a unit vector.
(For example, if given $\alpha$, we could define $\mc T = \alpha(\iota\otimes\xi_0)$).
Let $\tilde\alpha:H\rightarrow H\otimes K$ be the operator given by $\tilde\alpha(\xi)
= T(\xi\otimes\xi_0)$ for $\xi\in H$.  Then:
\begin{enumerate}
\item $\|\tilde\alpha\| = \|\alpha\|$;
\item $\tilde\alpha^* \tilde\alpha = \alpha^*\alpha \in \mc L(A)\cong M(A)$,
  where we identify $M(A)$ as a subalgebra of $\mc B(H)$.
\item $\tilde\alpha(a(\xi)) = e_\xi \alpha(a)$ for $a\in A$ and $\xi\in H$;
  so $\tilde\alpha$ depends only on $\alpha$ (and not $\xi_0$ or $\mc T$).
\end{enumerate}
\end{proposition}
\begin{proof}
Let $\Gamma:\mc K(A\otimes K) \rightarrow A\otimes\mc B_0(K) \subseteq \mc B(H\otimes K)$
be the isomorphism, which satisfies $\Gamma(\theta_{a\otimes\xi,b\otimes\eta}) =
ab^* \otimes \theta_{\xi,\eta}$ for $a,b\in A$ and $\xi,\eta\in K$.
Thus, for $c\in A$, $\phi\in H$ and $\gamma\in K$,
\[ \Gamma(\theta_{a\otimes\xi,b\otimes\eta})(c(\phi)\otimes\gamma)
= ab^*c(\phi) \otimes \xi (\eta|\gamma). \]
Also, $e_\phi (\theta_{a\otimes\xi,b\otimes\eta}(c\otimes\gamma) ) =
ab^*c(\phi) \otimes\xi (\eta|\gamma)$.  Let $\theta\in\mc K(A\otimes K),
\tau\in A\otimes K$ and $\phi\in H$.  So we have shown that
$e_\phi( \theta(\tau) ) = \Gamma(\theta)( e_\phi(\tau) )$.
By definition, we have that $\Gamma( \mc T \theta) = T \Gamma(\theta)$, and so
\[ e_\phi\big( \mc T \theta(\tau) \big) = \Gamma(\mc T\theta) \big( e_\phi(\tau) \big)
= T \Gamma(\theta) \big( e_\phi(\tau) \big) = T e_\phi\big(\theta(\tau)\big). \]
By density, it follows that
\[ e_\phi\big( \mc T(\tau) \big) = T e_\phi(\tau)
\qquad (\tau\in A\otimes K, \phi\in H). \]

So immediately we see that for $a\in A$ and $\xi\in H$,
\[ \tilde\alpha\big( a(\xi) \big) = T (a(\xi)\otimes\xi_0) =
T e_\xi (a\otimes\xi_0) = e_\xi\big( \mc T (a\otimes\xi_0) \big)
= e_\xi \alpha(a), \]
as claimed.  Then, for $a,b\in A$ and $\xi,\eta\in H$,
\begin{align*} \big( \tilde\alpha(a(\xi)) \big| \tilde\alpha(b(\eta)) \big)
&= \big( e_{\xi} \alpha(a) \big| e_\eta \alpha(b) \big)
= \big( \xi \big| (\alpha(a)|\alpha(b)) \eta \big)
= \big( \xi \big| a^* \alpha^*\alpha b \eta \big)
= \big( a(\xi) \big| \alpha^* \alpha b(\eta) \big).
\end{align*}
It follows that $\tilde\alpha^*\tilde\alpha$ agrees with $\alpha^*\alpha$
as operators on $H$.  Then $\|\alpha\|^2 = \|\alpha^*\alpha\| =
\|\tilde\alpha^*\tilde\alpha\| = \|\tilde\alpha\|^2$, finishing the proof.
\end{proof}

\begin{proposition}\label{prop:three}
Let $B$ be a C$^*$-algebra and let $\phi:A\rightarrow M(B)$ be a non-degenerate
$*$-homomorphism.
Let $\alpha\in\mc L(A,A\otimes K)$ and $\mc T\in\mc L(A\otimes K)$ be related by
$\alpha = \mc T(\iota\otimes\xi_0)^*$, where $\xi_0\in K$ is a unit vector.
Let $S = (\phi\otimes\iota)T \in M(B\otimes\mc B_0(K))$, use this to induce
$\mc S\in\mc L(B\otimes K)$, and then define $\phi*\alpha = \mc S(\iota\otimes\xi_0)^*
\in \mc L(B,B\otimes K)$.  Then:
\begin{enumerate}
\item $(\iota\otimes\xi) (\phi*\alpha) =
\phi( (\iota\otimes\xi)\alpha )$ for each $\xi\in K$;
\item $\phi*\alpha$ depends only upon $\alpha$;
\end{enumerate}
\end{proposition}
\begin{proof}
As before, let $\Gamma:\mc K(A\otimes K)\rightarrow A\otimes\mc B_0(K)$ be the
isomorphism, with strict extension $\tilde\Gamma$; we use the same notation for
the isomorphism $\mc K(B\otimes K)\rightarrow B\otimes\mc B_0(K)$.
Let $\phi_0$ be the following composition
\[ \xymatrix{ \mc K(A\otimes K) \ar[r]^\Gamma &
A \otimes \mc B_0(K) \ar[r]^{\phi\otimes\iota} &
M(B) \otimes \mc B_0(K) \ar@{^{(}->}[r] &
M(B\otimes\mc B_0(K)) \ar[r]^{\tilde\Gamma^{-1}} &
\mc L(B\otimes K), } \]
and let $\tilde\phi_0:\mc L(A\otimes K)\rightarrow \mc L(B\otimes K)$
be the strict extension.  Thus $\mc S = \tilde\phi_0(\mc T)$.
For $\xi\in K$, let
\[ y = (\iota\otimes\xi) \tilde\phi_0(\mc T) (\iota\otimes\xi_0)^*
\in M(B), \quad x = (\iota\otimes\xi)\mc T(\iota\otimes\xi_0)^* \in M(A). \]
To show (i), we are required to show that $\phi(x)=y$.
As $\phi$ is non-degenerate, this is equivalent to
$\phi(xa)b = y\phi(a)b$ for $a\in A,b\in B$, that is,
\[ \phi\big( (\iota\otimes\xi)\mc T(a\otimes\xi_0) \big) b
= (\iota\otimes\xi)\tilde\phi_0(\mc T)(\phi(a)b\otimes\xi_0) \qquad (a\in A,b\in B). \]

Now, for $a,c\in A$, $b\in B$ and $\eta,\gamma\in K$,
\[ \tilde\phi_0\big(\theta_{a\otimes\xi_0,c\otimes\eta}\big)(b\otimes\gamma)
= \Gamma^{-1}\Big((\phi\otimes\iota)\big( ac^*\otimes\theta_{\xi_0,\eta} \big)\Big) (b\otimes\gamma)
= \phi(ac^*)b \otimes \xi_0 (\eta|\gamma). \]
So also
\[ \tilde\phi_0(\mc T) \big( \phi(ac^*)b \otimes \xi_0 \big) (\eta|\gamma) 
= \tilde\phi_0(\mc T)\tilde\phi_0\big(\theta_{a\otimes\xi_0,c\otimes\eta}\big)(b\otimes\gamma)
= \tilde\phi_0\big(\theta_{\mc T(a\otimes\xi_0),c\otimes\eta}\big)(b\otimes\gamma). \]

It seems easier to use an approximation argument now.  For $\epsilon>0$,
we can find $\tau\in A\odot K$ with
\[ \tau = \sum_k a_k \otimes \xi_k, \quad
\big\| \mc T(a\otimes\xi_0) - \tau \big\|\leq\epsilon. \]
Then $\| \theta_{\mc T(a\otimes\xi_0),c\otimes\eta} - \theta_{\tau,c\otimes\eta} \|
\leq \epsilon \|c\| \|\eta\|$.  Thus the previous paragraph shows that
\[ \Big\| \tilde\phi_0(\mc T) \big( \phi(ac^*)b \otimes \xi_0 \big) (\eta|\gamma) -
\sum_k \phi(a_kc^*)b \otimes \xi_k(\eta|\gamma) \Big\|
\leq \epsilon\|c\|\|\eta\|\|b\|\|\gamma\|, \]
Letting $c$ run through an approximate identity for $A$, and choosing $\eta=\gamma$
to be a unit vector shows that
\[ \Big\| \tilde\phi_0(\mc T) \big( \phi(a)b \otimes \xi_0 \big) -
\sum_k \phi(a_k)b \otimes \xi_k \Big\|
\leq \epsilon \|b\|. \]
Thus also
\[ \Big\| (\iota\otimes\xi)\tilde\phi_0(\mc T) \big( \phi(a)b \otimes \xi_0 \big) -
\sum_k \phi(a_k)b \otimes (\xi|\xi_k) \Big\|
\leq \epsilon \|b\|\|\xi\|. \]
However, similarly
\[ \Big\| \phi\big( (\iota\otimes\xi)\mc T(a\otimes\xi_0) \big) b -
\sum_k \phi(a_k)b \otimes (\xi|\xi_k) \Big\| \leq \epsilon \|b\| \|\xi\|. \]
As $\epsilon>0$ was arbitrary, this completes the proof of (i).
It is immediate that (i) implies (ii).
\end{proof}

\begin{proposition}\label{prop:five}
With the notation of the previous proposition, suppose that $B$ is non-degenerately
represented on $H\otimes H$, and that for some $V\in\mc B(H\otimes H)$, we have
that $\phi(a) = V^*(1\otimes a)V$ for $a\in A$.  Then
$(\phi*\alpha)\tilde{\,} = V_{12}^*(1\otimes\tilde\alpha)V$.
\end{proposition}
\begin{proof}
Combining the two previous propositions, we see that $(\phi*\alpha)\tilde{\,}(\xi) =
S(\xi\otimes\xi_0)$ for $\xi\in H\otimes H$.  Now, clearly $S = V_{12}^* T_{23} V_{12}
\in \mc B(H\otimes H\otimes K)$, and so
\[ (\phi*\alpha)\tilde{\,}(\xi) = V_{12}^* T_{23} \big( V(\xi)\otimes\xi_0 \big)
= V_{12}^* (1\otimes\tilde\alpha) V(\xi) \qquad (\xi\in H\otimes H), \]
as required.
\end{proof}

\section{Left-multipliers}

Let $\G$ be a locally compact quantum group.
In this section, we prove a complete analogy of Gilbert's result, for
represented, completely bounded left multipliers of $L^1(\hat\G)$.

Let $K$ be a Hilbert space, and consider the Hilbert C$^*$-module $C_0(\G)\otimes K$.
We shall say that a pair $(\alpha,\beta)$ of maps in $\mc L(C_0(\G),C_0(\G)\otimes K)$
is \emph{invariant} if
\[ (1\otimes\beta)^*(\Delta*\alpha) \in \mc L(C_0(\G)\otimes C_0(\G))
= M(C_0(\G)\otimes C_0(\G)) \]
is really in $C^b(\G)\otimes 1$.  Here $\Delta:C_0(\G)\rightarrow
M(C_0(\G)\otimes C_0(\G))$ is non-degenerate, and so we can apply
Proposition~\ref{prop:three} to form $\Delta*\alpha \in \mc L(C_0(\G)\otimes C_0(\G),
C_0(\G) \otimes C_0(\G) \otimes K)$.

When $\G = G$ a locally compact group, then $\alpha,\beta\in C^b(G,K)$,
and $\Delta*\alpha\in C^b(G\times G,K)$.  For $\xi\in K$ and $s,t\in G$, we have
\[ \big( \xi \big| (\Delta*\alpha)(s,t) \big) = (\iota\otimes\xi)(\Delta*\alpha)(s,t)
= \Delta\big( (\iota\otimes\xi)\alpha \big)(s,t)
= \big( (\iota\otimes\xi)\alpha \big)(st)
= \big(\xi\big|\alpha(st)\big). \]
So $(\Delta*\alpha)(s,t) = \alpha(st)$, as we might hope.  Then $(\alpha,\beta)$
is an invariant pair if there exists $f\in C^b(G)$ with
\[ \big( \beta(t) \big| \alpha(st) \big) = f(s) \qquad (s,t\in G), \]
or equivalently, if $f(st^{-1}) = (\beta(t)|\alpha(s))$ for $s,t\in G$.
This is clearly equivalent, though not identical, to Gilbert's condition, as outlined
in the introduction.  Proposition~\ref{prop:two} below shows that it is no surprise
that the $f\in C^b(G)$ appearing from $(1\otimes\beta)^*(\Delta*\alpha)=f\otimes 1$
should be the multiplier given by the pair $(\alpha,\beta)$.

By Proposition~\ref{prop:five}, we see that,
equivalently, $(\alpha,\beta)$ is invariant if
\[ (1\otimes\tilde\beta)^* W_{12}^*(1\otimes\tilde\alpha)W \in C^b(\G)\otimes 1, \]
as operators on $\mc B(L^2(\G)\otimes L^2(\G))$.
Here we use that $(1\otimes\beta)\tilde{} = 1\otimes\tilde\beta$.

\begin{proposition}\label{prop:two}
Let $\alpha, \beta \in \mc L(C_0(\G),C_0(\G)\otimes K)$, and for
$x\in L^\infty(\hat\G)$, define $L(x) = \tilde\beta^*(x\otimes 1)\tilde\alpha$.
Let $a\in C^b(\G)$.  The following are equivalent:
\begin{enumerate}
\item $L$ is the adjoint a completely bounded left multiplier on $L^1(\hat\G)$
  represented by $a$;
\item the pair $(\alpha,\beta)$ is invariant, with $(1\otimes\beta)^*(\Delta*\alpha)
  = a \otimes 1$.
\end{enumerate}
\end{proposition}
\begin{proof}
Clearly $L$ is a normal completely bounded map $L^\infty(\hat\G)\rightarrow\mc B(L^2(\G))$.
As $\hat W = \sigma W^* \sigma$, we see that
\begin{align*} (L\otimes\iota)(\hat W) &= 
   (\tilde\beta^*\otimes 1) \hat W_{13} (\tilde\alpha\otimes 1)
= (\tilde\beta^*\otimes 1) \sigma_{13}W^*_{13}\sigma_{13} (\tilde\alpha\otimes 1) \\
&= \sigma (1\otimes\tilde\beta^*\sigma)W^*_{13}(1\otimes\sigma\tilde\alpha)\sigma 
= \sigma (1\otimes\tilde\beta^*)W^*_{12}(1\otimes\tilde\alpha)\sigma.
\end{align*}
So, if (2) holds, then
\[ (L\otimes\iota)(\hat W) = \sigma (a\otimes 1)W^* \sigma = (1\otimes a) \hat W. \]
By the (left) version of Proposition~\ref{prop:one}, it follows that (1) holds.

Conversely, if (1) holds, then by Proposition~\ref{prop:four}, we have that
$(L\otimes\iota)(\hat W) = (1\otimes a)\hat W$, which shows that (2) holds.
\end{proof}

Notice that we here assume that $a\in C^b(\G)$, while in Section~\ref{multdual}
we could only ensure that $a\in L^\infty(\G)$.  The next result clarifies this.

\begin{theorem}\label{thm:three}
Let $L_*\in\mc{CB}(L^1(\hat\G))$ and $a\in L^\infty(\G)$ be such that
$a \hat\lambda(\hat\omega) = \hat\lambda(L_*(\hat\omega))$ for $\hat\omega\in L^1(\hat\G)$.
There exists a Hilbert space $K$ and an invariant pair $(\alpha,\beta)$ of maps
in $\mc L(C_0(\G),C_0(\G)\otimes K)$ such that $(\alpha,\beta)$ induces $L = (L_*)^*$
as in Proposition~\ref{prop:two}, and with $\|\alpha\|\|\beta\|=\|L\|_{cb}$.
Furthermore, automatically $a\in C^b(\G)$, so $L_*$ is represented.
\end{theorem}
\begin{proof}
Let $L=L_*^*\in\mc{CB}(L^\infty(\hat\G))$.  As $L$ is normal, we can find a Hilbert
space $K$, a normal $*$-representation $\pi:L^\infty(\hat\G)\rightarrow\mc B(K)$ and
maps $P,Q:L^2(\G)\rightarrow K$ with $\|P\| \|Q\|=\|L\|_{cb}$, and with
\[ L(x) = Q^*\pi(x)P \qquad (x\in L^\infty(\hat\G)). \]
This is, of course, the usual representation result for completely bounded maps,
but as $L$ is normal, we can assume that $\pi$ is normal: the details of this
change are worked out in the proof of \cite[Theorem~2.4]{HM}, for example.

Kustermans showed in \cite[Corollary~4.3]{kus1} that if $B$ is a C$^*$-algebra and
$\phi:L^1_\sharp(\G)\rightarrow M(B)$ is a non-degenerate $*$-homomorphism
(in the sense that $\{\phi(\omega)b : \omega\in L^1_\sharp(\G), b\in B\}$ is
linearly dense in $B$), then
there is a unitary $U\in M(C_0(\G)\otimes B)$ such that
\[ \phi(\omega) = (\omega\otimes\iota)(U) \quad (\omega\in L^1_\sharp(\G))
\qquad (\Delta\otimes\iota)(U) = U_{13}U_{23}. \]
The philosophy here is that $\phi$ extends to a $*$-homomorphism from the enveloping
C$^*$-algebra of $L^1_\sharp(\G)$, and so $\phi$ can be thought of as a representation
of the (universal) quantum group $\hat\G$, whereas $U$ is a corepresentation of $\G$;
Kustermans's result is that there is a correspondence between representations of $\hat\G$
and corepresentations of $\G$.

As we may assume that $\pi:L^\infty(\hat\G)\rightarrow\mc B(K)$ is unital,
and $L^1_\sharp(\G)$ is dense in $L^1(\G)$, it follows that
$\pi\lambda: L^1_\sharp(\G) \rightarrow \mc B(K) = M(\mc B_0(K))$ is non-degenerate,
and so we can find a representing unitary $U\in M(C_0(\G)\otimes\mc B_0(K))$.
Notice that $C_0(\G)\otimes\mc B_0(K)$ acts non-degenerately on $L^2(\G)
\otimes K$, and so we may identify $U$ with an operator in the von Neumann
algebra $L^\infty(\G)\vnten\mc B(K)$.

Let $\omega\in L^1_\sharp(\G)$ and let $\gamma,\delta\in K$.  Then
\begin{align*} \ip{U}{\omega\otimes\omega_{\gamma,\delta}}
&= \big( \gamma \big| (\omega\otimes\iota)(U) \delta \big)
= \big( \gamma \big| \pi(\lambda(\omega))\delta \big)
= \ip{\lambda(\omega)}{\pi_*(\omega_{\gamma,\delta})} \\
&= \ip{(\omega\otimes\iota)(W)}{\pi_*(\omega_{\gamma,\delta})}
= \ip{\pi((\omega\otimes\iota)(W))}{\omega_{\gamma,\delta}} \\
&= \ip{(\omega\otimes\iota)(\iota\otimes\pi)(W)}{\omega_{\gamma,\delta}}
= \ip{(\iota\otimes\pi)(W)}{\omega\otimes\omega_{\gamma,\delta}}.
\end{align*}
Here $\pi_*:\mc B(K)_* \rightarrow L^1(\hat\G)$ is the pre-adjoint, which exists
as $\pi$ is normal.  By density of $L^1_\sharp(\G)$ in $L^1(\G)$, we conclude
that $U = (\iota\otimes\pi)(W) \in L^\infty(\G)\vnten\mc B(K)$.  Indeed,
if we wished, we could define $U$ this way, and avoid using \cite{kus1}.

Also, we identify $M(C_0(\G)\otimes\mc B_0(K))$ with $\mc L(C_0(\G)\otimes K)$
and so $U$ induces $\mc U\in \mc L(C_0(\G)\otimes K)$.
Similarly, $W\in M(C_0(\G)\otimes\mc B_0(L^2(\G)))$ is associated to
$\mc W\in\mc L(C_0(\G)\otimes L^2(\G))$.
Fix a unit vector $\xi_0\in L^2(\G)$ and define
\begin{align*} \alpha &= \mc U^*(1\otimes P)\mc W (\iota\otimes\xi_0)^* \in
   \mc L(C_0(\G),C_0(\G)\otimes K), \\
\beta &= \mc U^*(1\otimes Q)\mc W (\iota\otimes\xi_0)^* \in
   \mc L(C_0(\G),C_0(\G)\otimes K). \end{align*}
Notice that $\|\alpha\|\|\beta\|\leq\|P\|\|Q\|=\|L\|_{cb}$.
By Proposition~\ref{prop:six}, $\alpha$ induces $\tilde\alpha\in
\mc B(L^2(\G),L^2(\G)\otimes K)$, and similarly $\beta$ induces $\tilde\beta$,
and in fact, we have that
\[ \tilde\alpha(\xi) = U^*(1\otimes P)W(\xi\otimes\xi_0), \quad
\tilde\beta(\xi) = U^*(1\otimes Q)W(\xi\otimes\xi_0)
\qquad (\xi\in L^2(\G)). \]

We next show that $(\alpha,\beta)$ is invariant, for which we need to consider
$(1\otimes\tilde\beta)^* W_{12}^*(1\otimes\tilde\alpha)W$.  Let $\xi,\eta\in
L^2(\G)\otimes L^2(\G)$, and we calculate that
\begin{align*}
\big( (1\otimes\tilde\beta)\xi &\big| W_{12}^*(1\otimes\tilde\alpha)W\eta \big) \\
&= \big( U_{23}^*(1\otimes 1\otimes Q)W_{23}(\xi\otimes\xi_0) \big|
   W_{12}^* U_{23}^*(1\otimes 1\otimes P)W_{23}W_{12}(\eta\otimes\xi_0) \big) \\
&= \big( U_{23}^*(1\otimes 1\otimes Q)W_{23}(\xi\otimes\xi_0) \big|
   W_{12}^* U_{23}^*W_{12}(1\otimes 1\otimes P)W_{13}W_{23}(\eta\otimes\xi_0) \big)
\end{align*}
Here we used the Pentagonal relation $W_{12} W_{13} W_{23} = W_{23} W_{12}$.
Now, if $X\in L^\infty(\G)\vnten\mc B(K)$, then $W_{12}^* X_{23} W_{12}
= (\Delta\otimes\iota)X$, so we find that $W_{12}^* U_{23}^* W_{12}
= (\Delta\otimes\iota)(U^*) = U_{23}^* U_{13}^*$ as $\Delta$ is a $*$-homomorphism.
Thus we get
\begin{align*}
\big( (1\otimes\tilde\beta)\xi &\big| W_{12}^*(1\otimes\tilde\alpha)W\eta \big) \\
&= \big( U_{23}^*(1\otimes 1\otimes Q)W_{23}(\xi\otimes\xi_0) \big|
   U_{23}^* U_{13}^*(1\otimes 1\otimes P)W_{13}W_{23}(\eta\otimes\xi_0) \big) \\
&= \big( W_{23}(\xi\otimes\xi_0) \big| (1\otimes 1\otimes Q^*)(\iota\otimes\pi)(W^*)_{13}
   (1\otimes 1\otimes P)W_{13}W_{23}(\eta\otimes\xi_0) \big) \\
&= \big( W_{23}(\xi\otimes\xi_0) \big| (\iota\otimes L)(W^*)_{13}
   W_{13}W_{23}(\eta\otimes\xi_0) \big)
\end{align*}
By Proposition~\ref{prop:four}, $(L\otimes\iota)(\hat W) = (1\otimes a)\hat W$.
Equivalently, we have $(\iota\otimes L)(W^*) = (a\otimes 1)W^*$, so we get
\begin{align*}
\big( (1\otimes\tilde\beta)\xi &\big| W_{12}^*(1\otimes\tilde\alpha)W\eta \big)
&= \big( W_{23}(\xi\otimes\xi_0) \big| (a\otimes 1\otimes 1)W^*_{13}
   W_{13}W_{23}(\eta\otimes\xi_0) \big)
= \big( \xi \big| (a\otimes 1)\eta \big). \end{align*}
Thus $(\alpha,\beta)$ is invariant, and induces $a$; in particular, we must
have that $a\in C^b(\G)$.  So by Proposition~\ref{prop:two},
if $L_0^*(x) = \tilde\beta^*(x\otimes 1)\tilde\alpha$ for $x\in L^\infty(\hat\G)$,
then $L_0^*$ is normal, maps into $L^\infty(\hat\G)$, and the pre-adjoint $L_0$
satisfies $\hat\lambda(L_0(\hat\omega)) = a \hat\lambda(\hat\omega)$ for $\hat\omega
\in L^1(\hat\G)$.  As $\hat\lambda$ injects, it follows that $L_0 = L_*$,
as required.
\end{proof}

\section{Approaches to right multipliers}\label{ontheright}

In the previous section, we studied represented completely bounded left
multipliers.  There are a number of ways to deal with right multipliers:
\begin{itemize}
\item Directly try to generalise the proof of Proposition~\ref{prop:two}.
We do this in Proposition~\ref{prop:seven}.  However, there are no \emph{a priori}
links with $\mc L(C_0(\G),C_0(\G)\otimes K)$.
\item Use the unitary antipode to convert right multipliers into left multipliers.
We do this in Lemma~\ref{lem:ten}, which gives formulas suggestive of those
in Proposition~\ref{prop:seven}.  We are also now in a position to use
\cite[Corollary~4.4]{jnr} to show that every completely bounded left
multiplier is represented.
\item Use the opposite algebra $L^1(\hat\G^\op)$, as a right multiplier of
$L^1(\hat\G)$ is a left multiplier of $L^1(\hat\G^\op)$.  However, by the
duality theory, this leads us to consider the algebra $C_0(\G')$.  We find a
way to move back to $C_0(\G)$ which gives exactly the formulas we were led
to consider by Lemma~\ref{lem:ten}.  Further, we find that a pair
$(\alpha,\beta)$ in $\mc L(C_0(\G),C_0(\G)\otimes K)$ is invariant if and only
if $(\beta,\alpha)$ is invariant.  This ``swap'' operation $(\alpha,\beta)\mapsto
(\beta,\alpha)$ induces a natural map $L_*\mapsto L_*^\dagger$ of left multipliers,
see Proposition~\ref{prop:thirteen}.
\item To make links with \cite{kus1}, we consider a ``coordinate'' approach
in Section~\ref{sec:coord} which leads to Theorem~\ref{thm:one} which, in
particular, allows us to show that the map
$(\alpha,\beta)\mapsto (\beta,\alpha)$ \emph{is} the antipode (in a technical sense).
\end{itemize}




\begin{proposition}\label{prop:seven}
Let $P,Q \in \mc B(L^2(\G), L^2(\G)\otimes K)$, and
define a map $R:L^\infty(\hat\G)\rightarrow\mc B(L^2(\G))$ by
$R(x) = P^*(x\otimes 1)Q$ for $x\in L^\infty(\hat\G)$.
Let $a\in C^b(\G)$.  The following are equivalent:
\begin{enumerate}
\item $R$ is the adjoint of a completely bounded right multiplier of
  $L^1(\hat\G)$ which is represented by $a$;
\item $(1\otimes Q^*) W_{12} (1\otimes P) W^* = a^*\otimes 1$.
\end{enumerate}
\end{proposition}
\begin{proof}
As in the proof of Proposition~\ref{prop:two}, 
\[ (R\otimes\iota)(\hat W) = \sigma(1\otimes P^*)W_{12}^*(1\otimes Q)\sigma. \]
Thus, if (i) holds, then by Proposition~\ref{prop:four},
\[ \sigma(1\otimes P^*)W_{12}^*(1\otimes Q)\sigma
= \sigma W^*(a\otimes 1)\sigma. \]
Taking the adjoint gives (ii).  The converse follows from Proposition~\ref{prop:one}.
\end{proof}

Compared to Proposition~\ref{prop:two}, we have swapped $W$ with $W^*$.  As such,
it's not immediately clear how to relate $P$ and $Q$ to maps in
$\mc L(C_0(\G), C_0(\G)\otimes K)$.

Another approach to right multipliers is to use the unitary antipode
$\hat\kappa$ to convert the problem to studying left multipliers, which
follows, as $\hat\kappa_*$ is anti-multiplicative on $L^1(\hat\G)$.

\begin{lemma}\label{lem:ten}
Let $R_*:L^1(\hat\G)\rightarrow L^1(\hat\G)$ be a right multiplier, and defined
$L_* = \hat\kappa_* R_* \hat\kappa_*$, a left multiplier. Then:
\begin{enumerate}
\item $R_*$ is completely bounded if and only if $L_*$ is;
\item if $R_*$ is represented by $a\in C^b(\G)$, then $L_*$ is
represented by $\kappa(a) \in C^b(\G)$.
\end{enumerate}
\end{lemma}
\begin{proof}
For (i), suppose first that $R_*$ is completely bounded, so that
$R\in\mc{CB}(L^\infty(\hat\G))$.  For $x\in L^\infty(\hat\G)$,
we have that $\hat\kappa(x)=Jx^*J$, and so
\[ L(x) = \hat\kappa R \hat\kappa(x) = J R(Jx^*J)^* J
\qquad (x\in L^\infty(\hat\G)). \]
As $R$ is completely bounded, it admits a dilation-- compare with the proof
of Theorem~\ref{thm:three} above, but here we will assume that the normal
representation $\pi$ is an amplification (as we may, see
\cite[Chapter~IV, Theorem~5.5]{tak} for example).  So there exists a Hilbert space
$H$ and bounded maps $U,V:L^2(\G)\rightarrow H\otimes L^2(\G)$ such that
$R(x) = V^*(1\otimes x)U$ for $x\in L^\infty(\hat\G)$.  Thus
\[ L(x) = JU^*(1\otimes J)(1\otimes x)(1\otimes J)VJ \qquad (x\in L^\infty(\hat\G)), \]
showing that $L$, and hence also $L_*$, are completely bounded.  The
converse follows similarly.

For (ii), let $\hat\omega\in L^1(\hat\G)$, so that
\[ \hat\lambda(L_*(\hat\omega)) = \kappa\hat\lambda\big( R_*\hat\kappa_*(\hat\omega) \big)
= \kappa\big( \hat\lambda(\hat\kappa_*(\hat\omega)) a \big)
= \kappa(a) \kappa\hat\lambda\big( \hat\kappa_*(\hat\omega) \big)
= \kappa(a) \hat\lambda(\hat\omega), \]
using that $\kappa\hat\lambda = \hat\lambda\hat\kappa_*$.
Hence $L_*$ is represented by $\kappa(a)$, as required.
\end{proof}

Thus, if $R_*$ is a completely bounded right multiplier which is
represented, then $L_*=\hat\kappa_* R_* \hat\kappa_*$ is a represented left
multiplier, and hence admits an invariant pair $(\alpha,\beta)$ in
$\mc L(C_0(\G),C_0(\G)\otimes K)$.
Indeed, for $x\in L^\infty(\hat\G)$, we have that
$L(x) = \tilde\beta^* (x\otimes 1) \tilde\alpha$, so that
\begin{align*} R(x) &= \hat\kappa L \hat\kappa(x)
= J L(\hat\kappa(x))^* J
= J \tilde\alpha^* (\hat\kappa(x)^*\otimes 1) \tilde\beta J
= J \tilde\alpha^* (JxJ\otimes 1) \tilde\beta J \\
&= J \tilde\alpha^*(J\otimes J_K) (x\otimes 1) (J\otimes J_K)\tilde\beta J.
\end{align*}
Here $J_K$ is some involution on $K$: a conjugate linear isometry with
$J_K^2=1$ (we can always find such a map: just write $K$ as $\ell^2(I)$ for some
index set $I$).  This gives one way to link the maps $P$ and $Q$ appearing in
Proposition~\ref{prop:seven} above to maps $\alpha,\beta$ in
$\mc L(C_0(\G),C_0(\G)\otimes K)$.  Furthermore, the map $(J\otimes J_K)\tilde\alpha J$
will appear (in slightly different context) below in Lemma~\ref{lemma:sixnew}.

The following is an improvement upon \cite[Corollary~4.4]{jnr},
in that we can show that every left or right multiplier is represented by
an element of $C^b(\G)$, and not just $L^\infty(\G)$.

\begin{proposition}\label{prop:ten}
Any left or right completely bounded multiplier of $L^1(\hat\G)$ is
represented by an element of $C^b(\G)$.
\end{proposition}
\begin{proof}
Let $R_*$ be a completely bounded right multiplier of $L^1(\hat\G)$, and choose
$x\in L^\infty(\G')$ by Theorem~\ref{thm:ten} (that is, using \cite{jnr}) so that
$\hat\rho(\hat\omega) x = \hat\rho(R_*(\hat\omega))$ for $\hat\omega\in L^1(\hat\G)$.
By the definition of $\hat\rho$, we see that $\hat\lambda(\hat\omega) J\hat J x
\hat JJ = \hat\lambda(R_*(\hat\omega))$ for each $\hat\omega\in L^1(\hat\G)$.
Set $b = J\hat J x \hat JJ$, and let $L_* = \hat\kappa_* R_* \hat\kappa_*$ so by
(the proof of) Lemma~\ref{lem:ten}, $L_*$ is a completely bounded left multiplier
with $\kappa(b) \hat\lambda(\hat\omega) = \hat\lambda(L_*(\hat\omega))$ for each
$\hat\omega\in L^1(\hat\G)$.  From
Theorem~\ref{thm:three}, it follows that $\kappa(b)\in C^b(\G)$, and so also
$b\in C^b(\G)$.  Similarly, using the unitary antipode, a similar argument
gives the result for completely bounded left multipliers.
\end{proof}

\subsection{Using the opposite algebra}\label{usingopalg}

Recall the definition of the opposite quantum group $\hat\G^\op$ from
Section~\ref{sec:lcqg}.
Given a completely bounded right multiplier $R_*$ of $L^1(\hat\G)$, write
$R_*^\op$ for $R_*$ considered as a map on $L^1(\hat\G^\op)$, so that
$R_*^\op$ is a completely bounded left multiplier.

We now know that $R_*^\op$ is represented, say by $JbJ \in C^b(\G') = J C^b(\G) J$. 
By Theorem~\ref{thm:three}, we can find $\alpha', \beta'
\in \mc L(C_0(\G'), C_0(\G')\otimes K)$ such that the pair $(\alpha',\beta')$ is invariant
with respect to $JbJ$, that is, $(1\otimes\beta')^*(\Delta'*\alpha') = JbJ\otimes 1$,
and such that $R(x) = \tilde\beta'{}^*(x\otimes 1)\tilde\alpha'$ for
$x\in L^\infty(\G)$.

Recall the isomorphism $\Phi:C_0(\G')\rightarrow C_0(\G); a\mapsto \hat J J a J \hat J$.
Given $\alpha'\in\mc L(C_0(\G'),C_0(\G')\otimes K)$, we notice that
$(\Phi\otimes\iota)\alpha'\Phi^{-1}$ is in $\mc L(C_0(\G),C_0(\G)\otimes K)$.
However, this isomorphism does not interact well with forming $\Delta*\alpha$
or $\tilde\alpha$ (for example, we get nothing like Lemma~\ref{lemma:sixa} below).
Rather, we study another bijection between
$\mc L(C_0(\G),C_0(\G)\otimes K)$ and $\mc L(C_0(\G'),C_0(\G')\otimes K)$
which comes at the cost of choosing an involution $J_K$ on $K$, which the
bijection will depend upon.  However, the results below show that, as far as
multipliers are concerned, there is no dependence upon $J_K$.  From now on, fix some
involution $J_K$ on $K$.

\begin{lemma}\label{lemma:sixnew}
Define an anti-linear isomorphism $\theta:C_0(\G')\rightarrow C_0(\G);
a\mapsto J a J$.
For $\alpha'\in\mc L(C_0(\G'),C_0(\G')\otimes K)$, the map
$\alpha = (\theta\otimes J_K)\alpha'\theta^{-1}$ is in $\mc L(C_0(\G),C_0(\G)\otimes K)$.
Furthermore, we have that $\tilde\alpha = (J\otimes J_K) \tilde\alpha' J$.
\end{lemma}
\begin{proof}
First check that for $\tau,\sigma\in C_0(\G)\otimes K$, we have that
\[ \big( (\theta\otimes J_K)\tau \big| (\theta\otimes J_K)\sigma \big)
= J (\tau|\sigma) J. \]
Then, for $a,b\in C_0(\G)$,
\[ \big( \alpha(a) \big| \alpha(b) \big) =
J\big( \alpha'(JaJ) \big| \alpha'(JbJ) \big) J
= a^* J {\alpha'}^* \alpha' J b, \]
where here ${\alpha'}^* \alpha'\in C^b(\G')$, and so $J {\alpha'}^* \alpha' J
\in C^b(\G)$.  It follows that $\alpha$ is well-defined and bounded.  We can
similarly show that $\alpha^* = \theta {\alpha'}^* (\theta^{-1}\otimes J_K)$, so
in particular, $\alpha$ is adjointable.

Let $a\in C_0(\G), \xi\in L^2(\G)$ and $\eta\in K$.  With reference to
Proposition~\ref{prop:six}, we have that
$e_\xi(\theta\otimes J_K)(a\otimes\eta) = JaJ\xi \otimes J_k(\eta)
= (J\otimes J_k) e_{J\xi} (a\otimes\eta)$.  It follows that
\[ \tilde\alpha(a\xi) = e_\xi (\theta\otimes J_K)\alpha' \theta^{-1}(a)
= (J\otimes J_K) e_{J\xi} \alpha' (JaJ) = (J\otimes J_K) \tilde\alpha' (Ja\xi), \]
and so $\tilde\alpha = (J\otimes J_K) \tilde\alpha' J$.
\end{proof}


\begin{lemma}\label{lemma:sixa}
Let $\alpha',\beta'\in \mc L(C_0(\G'),C_0(\G')\otimes K)$ and
$\alpha,\beta \in \mc L(C_0(\G),C_0(\G)\otimes K)$ be associated as in the previous
lemma.  Then the pair $(\alpha',\beta')$ is invariant with respect to $JbJ\in C^b(\G')$
if and only if the pair $(\alpha,\beta)$ is invariant with respect to $b\in C^b(\G)$.
\end{lemma}
\begin{proof}
We have that
\begin{align*} (\Delta'*\alpha)\tilde{} &= (W')^*_{12} (1\otimes\tilde\alpha')W' \\
&= (J\otimes J\otimes J_K) W^*_{12} (J\otimes J\otimes J_K) (J\otimes J\otimes J_K)
  (1\otimes\tilde\alpha)(J\otimes J)(J\otimes J) W (J\otimes J) \\
&= (J\otimes J\otimes J_K) W^*_{12} (1\otimes\tilde\alpha) W (J\otimes J)
  = (J\otimes J\otimes J_K) (\Delta*\alpha)\tilde{\ } (J\otimes J).
\end{align*}
Hence $(\alpha',\beta')$ being invariant with respect to $JbJ$ is equivalent to
\begin{align*} JbJ\otimes 1 &= (1\otimes\tilde\beta'{}^*)
  (J\otimes J\otimes J_K) (\Delta*\alpha)\tilde{\ } (J\otimes J) \\
&= (J\otimes J) (1\otimes\tilde\beta^*) (\Delta*\alpha)\tilde{\ } (J\otimes J).
\end{align*}
By applying $J\otimes J$ to both sides, this is equivalent to
$(\alpha,\beta)$ being invariant with respect to $b$, as claimed.
\end{proof}

For $x\in L^\infty(\G)$, we have that $R(x) = \tilde\beta'{}^*(x\otimes 1)\tilde\alpha'$.
By using Lemma~\ref{lemma:sixnew}, we see that
\[ R(x) = J \tilde\beta^* (JxJ\otimes 1) \tilde\alpha J
= \hat\kappa\big( \tilde\alpha^* (Jx^*J\otimes 1) \tilde\beta \big)
= \hat\kappa\big( \tilde\alpha^* (\hat\kappa(x)\otimes 1) \tilde\beta \big)
\qquad (x\in L^\infty(\hat\G)). \]
So to make links with Lemma~\ref{lem:ten}, we are led to look at the pair $(\beta,\alpha)$.

\begin{proposition}\label{prop:eight}
Let $(\alpha,\beta)$ be an invariant pair in $\mc L(C_0(\G),C_0(\G)\otimes K)$,
and let $(\alpha',\beta')$ be the associated invariant pair in
$\mc L(C_0(\G'),C_0(\G')\otimes K)$.  Let $R_*^\op$ be the left multiplier of
$L^1(\hat\G^\op)$ induced by $(\alpha',\beta')$, and let $R_*$ (a right multiplier
of $L^1(\hat\G)$) be represented by
$a\in C^b(\G)$.  Then $(\beta,\alpha)$ is invariant with respect to $\kappa(a)$.
\end{proposition}
\begin{proof}
Form $R_*^\op$ using $(\alpha',\beta')$, so that $R_*$ is a completely bounded
right multiplier of $L^1(\G)$.  By Proposition~\ref{prop:ten}, $R_*$ is
represented, say by $a\in C^b(\G)$.  Let $L_* = \hat\kappa_* R_* \hat\kappa_*$,
so by Lemma~\ref{lem:ten}, $L_*$ is a left multiplier represented by $\kappa(a)$.
For $x\in L^\infty(\hat\G)$, we have that $\hat\kappa L\hat\kappa(x) = R(x)
= \hat\kappa\big( \tilde\alpha^*( \hat\kappa(x)\otimes1 ) \tilde\beta \big)$,
using the above calculation.  Hence $L(x) = \tilde\alpha^*(x\otimes 1)\tilde\beta$.
By Proposition~\ref{prop:two}, it follows that $(\beta,\alpha)$ is invariant
with respect to $\kappa(a)$.
\end{proof}

We now show what happens with the induced left multipliers of $L^1(\hat\G)$,
without reference to $L^1(\hat\G^\op$).  We first need a lemma:
remember that $\hat\lambda^\op$ is the homomorphism $L^1(\hat\G^\op)\rightarrow
C_0(\G')$.

\begin{lemma}\label{lem:eight}
For $\hat\omega\in L^1(\hat\G)$, we have that $\hat\lambda^\op(\hat\omega)
= J\hat J\hat\lambda(\hat\omega^*)^* \hat JJ$.
\end{lemma}
\begin{proof}
From \cite[Section~4]{kus2}, we have that
$W^\op = (\hat J\otimes\hat J)W(\hat J\otimes\hat J)$, and so by duality,
$\hat W^\op = (J\otimes J) \hat W (J\otimes J)$.
For $\hat\omega = \hat\omega_{\xi_0,\eta_0} \in L^1(\hat\G)$, we have that
\[ \ip{x}{\hat\omega_{J\xi_0,J\eta_0}} = (J\xi_0|xJ\eta_0)
= (\eta_0|Jx^*J\xi_0) = \ip{\hat\kappa(x)}{\hat\omega^*}
\qquad (x\in L^\infty(\hat\G)). \]
Thus, for $\xi,\eta\in L^2(\G)$, we have
\begin{align*} \big( \xi \big| \hat\lambda^\op(\hat\omega)\eta \big) &=
\big( \xi \big| (\hat\omega\otimes\iota)(\hat W^\op) \eta \big)
= \big( \xi_0 \otimes \xi \big| (J\otimes J)\hat W(J\otimes J)
   (\eta_0\otimes \eta) \big) \\
&= \big( \hat W (J\eta_0\otimes J\eta) \big| J\xi_0\otimes J\xi \big)
= \overline{ \big( J\xi_0\otimes J\xi \big| \hat W (J\eta_0\otimes J\eta) \big) } \\
&= \overline{ \big( J\xi \big| (\hat\omega_{J\xi_0,J\eta_0}\otimes\iota)(\hat W)
   J\eta \big) }
= \overline{ \big( J\xi \big| \hat\lambda(\hat\kappa_*(\hat\omega^*)) J\eta \big) } \\
&= \big( \kappa\hat\lambda(\hat\omega^*) J\eta \big| J\xi \big)
= \big( \hat J \hat\lambda(\hat\omega^*)^* \hat JJ\eta \big| J\xi \big)
= \big( \hat JJ\xi \big| \hat\lambda(\hat\omega^*)^* \hat JJ\eta \big).
\end{align*}
Thus $\hat\lambda^\op(\hat\omega) = J \hat J \hat\lambda(\hat\omega^*)^* \hat JJ$.
\end{proof}

Given a left multiplier $L_*$ of $L^1(\hat\G)$ define
\[ L_*^\dagger(\hat\omega) = L_*(\hat\omega^*)^* \qquad (\hat\omega\in L^1(\G)). \]
For $\hat\omega\in L^1(\hat\G)$, recall that $\hat\omega^*\in L^1(\hat\G)$
satisfies $\ip{x}{\hat\omega^*} = \overline{\ip{x^*}{\hat\omega}}$ for
$x\in L^\infty(\hat\G)$.
As the coproduct $\Delta$ is a $*$-homomorphism, it is easy to see that
$L^1(\hat\G)\rightarrow L^1(\hat\G); \hat\omega \mapsto \hat\omega^*$ is
a conjugate-linear algebra homomorphism.  It follows that $L_*^\dagger$ is
a left multiplier; completely bounded if $L_*$ is (compare with the proof
of Lemma~\ref{lem:ten}).
Similarly, we define $R_*^\dagger$ for a right multiplier.

\begin{proposition}\label{prop:thirteen}
For $L_*\in M_{cb}^l(L^1(\G))$, let $L_*$ be given by an invariant
pair $(\alpha,\beta)$.  Then the invariant pair $(\beta,\alpha)$
induces the left multiplier $L^\dagger_*$.
\end{proposition}
\begin{proof}
Let $(\alpha,\beta)$ be invariant with respect to $b\in C^b(\G)$, and let
$(\beta,\alpha)$ be invariant with respect to $\kappa(a)$.  Thus
$(\beta',\alpha')$ is invariant with respect to $J\kappa(a)J = J\hat J a^* \hat JJ$.
Let $T_*^\op$ be the associated left multiplier of $L^1(\hat\G^\op)$,
and let $T_*$ be the associated right multiplier of $L^1(\hat\G)$.  Then,
as in Proposition~\ref{prop:eight}, we have that
\[ T(x) = (\tilde\alpha')^*(x\otimes 1)\tilde\beta'
= J\tilde\alpha^*(JxJ\otimes 1)\tilde\beta J
= \hat\kappa L \hat\kappa(x) \qquad (x\in L^\infty(\hat\G)). \]
It follows that
\[ \hat\lambda^\op(\hat\kappa_*L_*\hat\kappa_*(\hat\omega))
= \hat\lambda^\op(T_*^\op(\hat\omega))
= J \hat J a^* \hat JJ \hat\lambda^\op(\hat\omega) \qquad (\hat\omega\in L^1(\hat\G)). \]
Now, for $\hat\omega \in L^1(\hat\G)$, by Lemma~\ref{lem:eight}, we have that
$\hat\lambda^\op(\hat\kappa_*(\hat\omega)) =
J\hat J \hat\lambda(\hat\kappa(\hat\omega^*))^* \hat JJ
= J\hat J \kappa (\hat\lambda(\hat\omega^*))^* \hat JJ
= J \hat\lambda(\hat\omega^*) J$.
For $\hat\omega\in L^1(\hat\G)$, let $\hat\sigma = \hat\kappa_*(\hat\omega^*)$,
so also $\hat\omega = \hat\kappa_*(\hat\sigma^*)$.  Then
\[ \hat\lambda^\op(\hat\kappa_*L_*\hat\kappa_*(\hat\omega))
= J \hat\lambda(L_*\hat\kappa_*(\hat\omega)^*) J
= J \hat\lambda(L_*^\dagger\hat\kappa_*(\hat\omega^*)) J
= J \hat\lambda(L_*^\dagger(\hat\sigma)) J, \]
and also
\[ J\hat J a^* \hat JJ \hat\lambda^\op(\hat\omega)
= J\kappa(a)J \hat\lambda^\op(\hat\kappa_*(\hat\omega^*))
= J\kappa(a)J J \hat\lambda(\hat\omega) J. \]
As these two are equal, we see that
\[ \hat\lambda(L_*^\dagger(\hat\sigma)) = \kappa(a) \hat\lambda(\hat\sigma)
\qquad (\hat\sigma\in L^1(\hat\G)). \]
Thus $L_*^\dagger$ is represented by $\kappa(a)$, which $(\beta,\alpha)$ is
invariant with respect to, as required.
\end{proof}

\subsection{Taking a coordinate approach}\label{sec:coord}

We have shown that an invariant pair $(\alpha,\beta)$, say represented by $b\in C^b(\G)$,
gives rise to another invariant pair $(\beta,\alpha)$, say represented by
$\kappa(a)\in C^b(\G)$.  In this section, we show that the relationship between
$a$ and $b$ is given by the (in general, unbounded) antipode $S$.

Let us recall from \cite[Section~5.5]{kus}
that $MC_I(C_0(\G))$ is the collection of $(x_i)_{i\in I} \subseteq M(C_0(\G))$
such that $\sum_i x_i^*x_i$ is strictly convergent in $M(C_0(\G))$.  Similarly,
define $MR_I(C_0(\G))$ to be the collection of those families $(x_i^*)_{i\in I}$
with $(x_i)\in MC_I(C_0(\G))$.

Let $K$ be a Hilbert space, and let $\alpha\in\mc L(C_0(\G),C_0(\G)\otimes K)$.
Let $(e_i)$ be an orthonormal basis for $K$, and let $\alpha_i = (\iota\otimes e_i)\alpha
\in \mc L(C_0(\G)) \cong C^b(\G)$ for each $i$.  A simple calculation shows
that $(\iota\otimes e_i)^*(\iota\otimes e_i) = 1\otimes\theta_{e_i,e_i} \in
\mc L(C_0(\G)\otimes K)$, and so
$\sum_i (\iota\otimes e_i)^*(\iota\otimes e_i)$ converges strictly to the identity.
Thus $\sum_i \alpha_i^* \alpha_i$ converges strictly to $\alpha^*\alpha$, and
so $(\alpha_i) \in MC_I(C_0(\G))$.  Furthermore, we have that
\[ \alpha(a) = \sum_i \alpha_i a \otimes e_i \in C_0(\G)\otimes K
\qquad (a\in C_0(\G)), \]
with the sum converging in norm.

Similarly, from Proposition~\ref{prop:three}, we have that
$(\Delta*\alpha)_i = \Delta(\alpha_i)$ for all $i$.  Hence, a pair $(\alpha,\beta)$
is invariant with respect to $b\in C^b(\G)$ precisely when
\[ \sum_i (1\otimes\beta_i^*) \Delta(\alpha_i) = b\otimes 1\in C^b(\G)\otimes 1. \]

\begin{theorem}\label{thm:one}
For $a,b\in C^b(\G)$, the following are equivalent:
\begin{enumerate}
\item there is $R_* \in M_{cb}^r(L^1(\hat\G))$ represented by $a$, with
  $R_*^\op$ being represented by $JbJ$;
\item there is a pair $(\alpha,\beta)$ of maps in $\mc L(C_0(\G),C_0(\G)\otimes K)$
  which is invariant with respect to $b$, and with $(\beta,\alpha)$ being invariant
  with respect to $\kappa(a)$;
\item there is $L_*\in M_{cb}^l(L^1(\hat\G))$ represented by $b\in C^b(\G)$, with
  $L^\dagger_*$ being represented by $\kappa(a)$.
\end{enumerate}
Furthermore, if these hold, then $a\in D(S)^* = D(S^{-1}) = D(\tau_{-i/2})^*$
and $b = \tau_{-i/2}(a^*) = \hat J S^{-1}(a) \hat J$.
\end{theorem}
\begin{proof}
By Proposition~\ref{prop:eight}, (i) and (ii) are equivalent, and by
Proposition~\ref{prop:thirteen}, (ii) and (iii) are equivalent.

We shall assume (ii).  As $(\beta,\alpha)$ is invariant with respect to $\kappa(a)$,
applying the adjoint shows that
\[ \sum_i \Delta(\beta_i^*) (1\otimes\alpha_i) = \hat J a\hat J\otimes 1
\in C^b(\G)\otimes 1. \]
By \cite[Corollary~5.34]{kus} (and, as we are working with $C^b(\G)$ and not
$C_0(\G)$ here, we need also to look at \cite[Remark~5.44]{kus}) it follows that
$\kappa(a^*) \in D(S)$ with $S\kappa(a^*) = b$.  Thus $\tau_{-i/2}(a^*)=b$,
as claimed.
\end{proof}

For each $\hat\omega\in L^1(\hat\G)$, we have that $\hat\lambda(\hat\omega^*)^*
\in D(S) = D(\tau_{-i/2})$ and $S(\hat\lambda(\hat\omega^*)^*) = \hat\lambda(\hat\omega)$.
Furthermore, $\{ \hat\lambda(\hat\omega^*)^* : \hat\omega\in L^1(\hat\G) \}$ forms
a core for $S$ (either as an operator on $C_0(\G)$ or on $L^\infty(\G)$).
These results follow easily from \cite[Proposition~8.3]{kus} and
\cite[Proposition~2.4]{kus2}.  Combined with the work of Kustermans in \cite{kus4} on
strict extensions of one-parameter groups on C$^*$-algebras, these observations
would give another way to show the above theorem.  The proof of Lemma~\ref{lem:eight}
can be adapted to show that $\hat\lambda^{\op}(\hat\omega) = J\hat J
S^{-1}(\hat\lambda(\hat\omega))\hat JJ$ for $\hat\omega\in L^1(\hat\G)$, and this
could then be used to argue purely at the level of
multipliers, instead of with invariant pairs.

Notice that the ``coordinate'' approach is very close in spirit to how Vaes and Van Daele
gave a definition of a \emph{Hopf C$^*$-algebra} in \cite{vvd}.  It would be interesting
to explore this further, together with the implicit link with Haagerup tensor products
(which Spronk used extensively in his study of the completely bounded multipliers of
$A(G)$ in \cite{spronk}).  Indeed, if one looks at the proof of
\cite[Corollary~4.4]{jnr}, then there are two steps.  Firstly, the adjoint of
a right multiplier is extended from $L^\infty(\G)$ to a map on $B(L^2(\G))$ with
certain commutation properties (see \cite[Proposition~4.3]{jnr}) and then an argument
using the extended (or weak$^*$) Haagerup tensor product is used,
\cite[Proposition~3.2]{jnr} (compare with \cite[Theorem~4.2]{BS}, where it is more
explicit as to how the Haagerup tensor product appears).
Indeed, with this perspective, what we have done is to finesse where we can take
the elements in the extended Haagerup tensor product expansion (that is, from
$C^b(\G)$ and not $L^\infty(\G)$).  We note that \cite[Corollary~5.6]{spronk}
shows that in the motivating example of $A(G)$, we can even work with
$\operatorname{wap}(G)$
and not $C^b(G)$: it's unclear what the ``quantum'' analogue of this would be.

We curiously get the following strengthening of \cite[Corollary~5.34]{kus}
(and \cite[Remark~5.44]{kus}) where it is a hypothesis that there exists $b\in C^b(\G)$
with $b\otimes 1 = \sum_i (1\otimes p_i)\Delta(q_i)$, and the conclusion is that
$b=\overline S(a)$.
To be careful, we now do not identify $S$ with its strict closure.

\begin{corollary}
Let $a\in C^b(\G)$ be such that for some $(p_i)\in MR_I(C_0(\G))$ and
$(q_i)\in MC_I(C_0(\G))$, we have that
\[ a\otimes 1 = \sum_i \Delta(p_i)(1\otimes q_i). \]
Let $\overline S$ be the strict closure of $S$ on $C^b(\G)$.
Then $a\in D(\overline S)$ and
\[ \overline S(a)\otimes 1 = \sum_i (1\otimes p_i)\Delta(q_i). \]
\end{corollary}
\begin{proof}
Let $(q_i)$ and $(p_i^*)$ induce, respectively, $\alpha$ and $\beta$ in
$\mc L(C_0(\G),C_0(\G)\otimes\ell^2(I))$, so that by applying the adjoint,
we see that $(\beta,\alpha)$ is invariant with respect to $a^*$.  Then
$(\alpha,\beta)$ is invariant say with respect to $b\in C^b(\G)$.  Thus
\[ b\otimes 1 = \sum_i (1\otimes\beta_i^*)\Delta(\alpha_i)
= \sum_i (1\otimes p_i)\Delta(q_i). \]
By \cite[Remark~5.44]{kus}, or from Theorem~\ref{thm:one}, it follows that
$a\in D(\overline S)$ and $\overline S(a)=b$, as required.
\end{proof}

A slight subtly here is the following.  Suppose that actually $a\in C_0(\G)$,
so that the above theorem tells us that $a\in D(\overline S)$.  However, this is
seemingly not enough to ensure that $a\in D(S)$ (where $S$ is considered as a densely
defined operator on $C_0(\G)$).  Indeed, using that $S=\kappa \tau_{-i/2}$,
by \cite[Proposition~2.15]{kus4}, we have that $a\in D(S)$ if and only if 
$S(a) = b \in C_0(\G)$ (as $\kappa$ leaves $C_0(\G)$ invariant).  It is not clear to
us whether this is likely to be true or not.

We could have used this ``coordinate'' approach to $\mc L(C_0(\G),C_0(\G)
\otimes K)$ throughout.  However, this would have been much harder to motivate from
Gilbert's theorem.  Furthermore, in Section~\ref{cstarmods} above, we used that
$\mc L(C_0(\G),C_0(\G)\otimes K)$ was a ``slice'' of $\mc L(C_0(\G)\otimes K)$.
This seemed like a technical tool, but in the next two sections, we shall see how this
viewpoint actually appears quite natural and profitable.

\section{Links with universal quantum groups}\label{sec:links_with_uni}

For a locally compact group $G$, we always have that $B(G)$, the Fourier-Stieltjes
algebra of $G$, embeds into $M_{cb}A(G)$.  Furthermore, we can construct the maps
$\alpha,\beta$ in the Gilbert representation by using unitary representations of $G$.

An analogous result holds for quantum groups.  Firstly, we consider the analogue
of $B(G)$.  Given a locally compact quantum group $\G$, we can consider
the Banach $*$-algebra $L^1_\sharp(\G)$, and then take its universal enveloping
$C^*$-algebra, say $C_0^u(\hat\G)$.  In \cite{kus1}, it is shown that $C_0^u(\hat\G)$
admits a coproduct, left and right invariant weights, and so forth, all of these
objects interacting very well with the natural quotient map $\hat\pi:C_0^u(\hat\G)
\rightarrow C_0(\hat\G)$.  Indeed, we call $C_0^u(\hat\G)$ the \emph{universal
quantum group} of $\hat\G$, the essential difference with the \emph{reduced quantum
group} $C_0(\hat\G)$ being that the invariant weights are no longer faithful.
This is a generalisation of the difference between $C^*(G)$ and $C^*_r(G)$ for a
non-amenable locally compact group $G$.  Then $C_0^u(\hat\G)^*$ becomes a Banach
algebra, and $\hat\pi^*:M(\hat\G)=C_0(\hat\G)^* \rightarrow C_0^u(\hat\G)^*$
a homomorphism.

We showed in \cite{dawsm}, adapting the argument given in
\cite[page~914]{kus}, that $C_0^u(\hat\G)^*$ embeds into $M_{cb}L^1(\hat\G)$.
To be precise, let $\iota:L^1(\hat\G) \rightarrow C_0^u(\hat\G)^*$ be the natural
inclusion, given by composing the map $L^1(\hat\G)\rightarrow C_0(\hat\G)^*$ with $\hat\pi^*$.
Then \cite[Proposition~8.3]{dawsm} shows that $\iota(L^1(\hat\G))$ is an ideal in
$C_0^u(\hat\G)^*$ and that the induced map $C_0^u(\hat\G)^*\rightarrow
M_{cb}(L^1(\hat\G))$ is an injection.

If $L^1(\hat\G)$ has a bounded approximate identity (that is,
$\hat\G$ is \emph{coamenable}) then $M_{cb}(L^1(\hat\G)) = C_0^u(\hat\G)^* = M(\hat\G)$.
We remark that we don't know if the converse is true or not.
In particular, in the commutative case, for a locally compact group, $L^1(G)$
always has a bounded approximate identity, and so $M_{cb}(L^1(G))=M(G)$
(which is the classical Wendel's Theorem).  The following result
thus shows how measures in $M(G)$ arise from invariant pairs in
$\mc L(C^*_r(G),C^*_r(G)\otimes K)$ for a suitable Hilbert space $K$.

\begin{theorem}
There exists a Hilbert space $K$ with an involution $J_K$, and a unitary
$\mc U\in\mc L(C_0(\G)\otimes K)$ with the following property.
For each $\mu\in C^u_0(\hat\G)^*$, say giving a multiplier
$(L_*,R_*)\in M_{cb}(L^1(\hat\G))$, there exist $\xi_0,\eta_0\in K$ with
$\|\xi_0\| \|\eta_0\|=\|\mu\|$, and such that:
\begin{enumerate}
\item with $\alpha = \mc U^*(\iota\otimes\xi_0)^*$ and
  $\beta = \mc U^*(\iota\otimes\eta_0)^*$, we have that $(\alpha,\beta)$
  is an invariant pair which gives $L_*$;
\item with $\gamma=\mc U^*(\iota\otimes J_K\eta_0)^*$ and
  $\delta=\mc U^*(\iota\otimes J_K\xi_0)^*$, we have that $(\gamma,\delta)$ is invariant,
  and gives $\hat\kappa_*R_*\hat\kappa_*$ (and thus, using Section~\ref{ontheright},
  gives $R_*$).
\end{enumerate}
\end{theorem}
\begin{proof}
Let $\theta:C^u_0(\hat\G)\rightarrow\mc B(K)$ be the universal representation.
That is, for each state $\mu\in C_0^u(\hat\G)^*$, let $(H_\mu,\theta_\mu,\xi_\mu)$ be the
cyclic GNS construction for $\mu$, and let $K = \bigoplus_\mu H_\mu$ with $\theta$
the direct sum representation.

We next find our unitary $\mc U$.
Let $\lambda_u:L^1_\sharp(\G)\rightarrow C_0^u(\hat\G)$ be the natural map.
As in the proof of Theorem~\ref{thm:three}, using \cite{kus1}, as the map $L^1_\sharp(\G)
\rightarrow M(\mc B_0(K)); \omega\mapsto\theta(\lambda_u(\omega))$ is a
non-degenerate $*$-representation, there is a unitary corepresentation
$U\in M(C_0(\G)\otimes\mc B_0(K))$ with
\[ \theta(\lambda_u(\omega)) = (\omega\otimes\iota)(U) \quad (\omega\in L^1_\sharp(\G))
\qquad (\Delta\otimes\iota)(U) = U_{13}U_{23}. \]
Then $U$ induces $\mc U\in\mc L(C_0(\G)\otimes K)$.

Actually, the unitary $U$ is actually given by a ``universal'' unitary
$\hat U \in M(C_0(\G)\otimes C_0^u(\hat\G))$, by which we mean satisfies
$U = (\iota\otimes\theta)(\hat U)$, see the proof of \cite[Corollary~4.3]{kus1}.
Kustermans works on the dual side in \cite{kus1}, but as explained on
\cite[Page~311]{kus1}, we can use biduality to recover results for
$C_0^u(\hat\G)$.  In particular, $\hat U$ induces the coproduct in
the sense that
\[ (\hat\pi\otimes\iota)\big( \sigma(\hat\Delta_u(y)) \big)
= \hat U (\hat\pi(y)\otimes 1)\hat U^* \qquad (y\in C_0^u(\hat\G)). \]

Define $(\alpha,\beta)$ as in (i), where we choose $\xi_0$ and $\eta_0$ so that
$\omega_{\eta_0,\xi_0}\circ\theta = \mu$.
We then have that $\tilde\alpha(\xi) = U^*(\xi\otimes\xi_0)$
and $\tilde\beta(\xi) = U^*(\xi\otimes\eta_0)$, for $\xi\in L^2(\G)$.  Then
\begin{align*} (1\otimes\tilde\beta^*)W_{12}^*(1\otimes\tilde\alpha)W
&= (\iota\otimes\iota\otimes\omega_{\eta_0,\xi_0})
   \big( U_{23} W_{12}^* U_{23}^* W_{12} \big)
= (\iota\otimes\iota\otimes\omega_{\eta_0,\xi_0})
   \big( U_{23} (\Delta\otimes\iota)(U)^* \big) \\
&= (\iota\otimes\iota\otimes\omega_{\eta_0,\xi_0})\big( U_{23} U_{23}^* U_{13}^* \big)
= (\iota\otimes\iota\otimes\omega_{\eta_0,\xi_0})\big( U_{13}^* \big) \in
C^b(\G)\otimes 1, \end{align*}
as $U\in M(C_0(\G)\otimes\mc B_0(K))$, so the right slice of $U$ is in $M(C_0(\G))=C^b(\G)$.
Thus $(\alpha,\beta)$ is an invariant pair, inducing $L'_*\in\mc{CB}(L^1(\hat\G))$, say.

We wish to show that $L'_*$ is given by left multiplication by $\mu$.
Let $\hat\omega=\hat\omega_{\eta_1,\xi_1}\in L^1(\hat\G)$, so that
$\mu\iota(\hat\omega) \in \iota(L^1(\hat\G))$.
Let $\omega\in L^1_\sharp(\G)$, and set $x=\lambda(\omega)
\in C_0(\hat\G)$.  Then $\hat\pi(\lambda_u(\omega)) = x$, so
\begin{align*} \ip{x}{\iota^{-1}\big( \mu\iota(\hat\omega) \big)}
&= \ip{\mu\iota(\hat\omega)}{\lambda_u(\omega)}
= \ip{\mu \otimes \iota(\hat\omega)}{\hat\Delta_u(\lambda_u(\omega))}
= \ip{\iota(\hat\omega)\otimes \mu}{\sigma\hat\Delta_u(\lambda_u(\omega))} \\
&= \ip{\hat\omega \otimes \mu}{(\hat\pi\otimes\iota)(\sigma\hat\Delta_u(\lambda_u(\omega)))}
= \ip{\hat\omega \otimes \mu}{\hat U(\hat\pi(\lambda_u(\omega))\otimes1)\hat U^*} \\
&= \ip{\hat\omega\otimes\omega_{\eta_0,\xi_0}}{U(x\otimes1)U^*}
= \big( U^*(\eta_1\otimes\eta_0) \big| (x\otimes 1) U^*(\xi_1\otimes\xi_0) \big) \\
&= \big( \tilde\beta(\eta_1) \big| (x\otimes1)\tilde\alpha(\xi_1) \big)
= \ip{L(x)}{\hat\omega} = \ip{x}{L_*(\hat\omega)},
\end{align*}
as we hoped.  By density, this holds for all $x\in L^\infty(\hat\G)$, so that
$L'_* = L_*$ as required to show (i).

We next define $J_K$.  By \cite[Proposition~7.2]{kus1}, there is an anti-$^*$-automorphism
$\hat\kappa_u:C_0^u(\hat\G)\rightarrow C_0^u(\hat\G)$ which ``lifts'' $\hat\kappa$,
in the sense that $\hat\pi\hat\kappa_u = \hat\kappa\hat\pi$.  For each state
$\mu\in C_0^u(\hat\G)^*$,
let $\mu' = \hat\kappa_u^*(\mu)$, which is still a state, as $\kappa_u^*$ is an
anti-$*$-automorphism.  On each $H_\mu$, (densely) define $J_K$ by
\[ J_K\big( \theta_\mu(a) \xi_\mu \big) = \theta_{\mu'}( \hat\kappa_u(a^*) ) \xi_{\mu'}
\qquad (a\in C_0^u(\hat\G)). \]
Then, for $a\in C_0^u(\hat\G)$, we have
\[ \big\| J_K\big( \theta_\mu(a) \xi_\mu \big) \big\|^2 =
\ip{\mu'}{\hat\kappa_u(a)\hat\kappa_u(a^*)}
= \ip{\hat\kappa_u^*(\mu)}{\hat\kappa_u(a^*a)} = \ip{\mu}{a^*a}
= \big\| \theta_\mu(a) \xi_\mu \big\|^2. \]
Thus $J_K$ extends by linearity and continuity to all of $K$.  Clearly $J_K$
is an involution.  Then, for $a,b\in C_0^u(\hat\G)$, we have
\begin{align*}
J_K \theta(a^*) J_K \theta_\mu(b) \xi_\mu &=
J_K \theta_{\mu'}\big(a^* \hat\kappa_u(b^*)\big) \xi_{\mu'}
= \theta_\mu\big( \hat\kappa_u(\hat\kappa_u(b)a) \big) \xi_\mu \\
&= \theta_\mu\big( \hat\kappa_u(a) b \big) \xi_\mu
= \theta\big( \hat\kappa_u(a) \big) \theta_\mu(b) \xi_\mu.
\end{align*}
It follows that $\theta\hat\kappa_u(a) = J_K \theta(a^*) J_K$ for each
$a\in C^u_0(\hat\G)$.

Now define $(\gamma,\delta)$ as in (ii), so by the argument just given,
$(\gamma,\delta)$ is an invariant pair which induces the left multiplier given by
multiplication by $\omega_{J_K\xi_0,J_K\eta_0} \circ \theta \in C_0^u(\hat\G)^*$.
Now, for $x\in C_0^u(\hat\G)$,
\[ \ip{\omega_{J_K\xi_0,J_K\eta_0} \circ \theta}{x}
= (J_K\xi_0|\theta(x)J_K\eta_0) = (\eta_0|J_K\theta(x)^*J_K\xi_0)
= (\eta_0|\theta(\hat\kappa_u(x))\xi_0)
= \ip{\mu}{\hat\kappa_u(x)}. \]
Thus $(\gamma,\delta)$ gives the left multiplier induced by $\hat\kappa_u^*(\mu)$.
For $\hat\omega\in L^1(\hat\G)$, we have that $\iota(\hat\omega)\mu =
\iota(R_*(\hat\omega))$, and so
\[ \iota\big( \hat\kappa_* R_* \hat\kappa_*(\hat\omega) \big)
= \hat\kappa_u^* \iota\big( R_* \hat\kappa_*(\hat\omega) \big)
= \hat\kappa_u^* \big( \iota(\hat\kappa_*(\hat\omega)) \mu \big)
= \hat\kappa_u^*(\mu) \iota(\hat\omega). \]
Thus $(\gamma,\delta)$ gives $\hat\kappa_* R_* \hat\kappa_*(\hat\omega)$, showing (ii).
\end{proof}

Consider further $(\gamma,\delta)$ as in (ii) above.
By \cite[Proposition~7.2]{kus1} we have that $(\kappa\otimes\hat\kappa_u)(\hat U)=\hat U$.
As $U=(\iota\otimes\theta)(\hat U)$ and  $\theta\hat\kappa_u(\cdot)
= J_K \theta(\cdot)^* J_K$, we see that
\[ U = (\kappa\otimes\theta\hat\kappa_u)(\hat U)
= (J\otimes J_K) U^* (J\otimes J_K). \]
Now, we have that $\tilde\gamma(\xi) = U^*(\xi\otimes J_K\eta_0)$ for $\xi\in L^2(\G)$.
It follows that
\[ (J\otimes J_K) \tilde\gamma(\xi) = U(J\xi \otimes \eta_0) \qquad (\xi\in L^2(\G)), \]
and a similar formula holds for $\tilde\delta$.  Thus $\tilde\gamma$ and
$\tilde\delta$ are given by right slices of $U$; however, it is not clear what,
if any, meaning we can give to taking a right slice of $\mc U$.

\section{For two-sided multipliers}

In this final section, we look at two-sided multipliers.  Firstly, as we saw in
Section~\ref{multdual}, a two-sided multiplier $(L_*,R_*)\in M_{cb}(L^1(\hat\G))$
gives rise to represented multipliers, represented by the same $a\in C^b(\G)$.

Let $(L_*,R_*)\in M_{cb}(L^1(\hat\G))$, and recall the definitions of $L^\dagger_*$
and $R^\dagger_*$ from Section~\ref{usingopalg}.
For $\hat\omega,\hat\sigma\in L^1(\hat\G)$ we have that
\[ \hat\omega L_*^\dagger(\hat\sigma) = \big( \hat\omega^* L_*(\hat\sigma^*) \big)^*
= \big( R_*(\hat\omega^*) \hat\sigma^* \big)^* = R_*^\dagger(\hat\omega) \hat\sigma. \]
Thus the map $(L_*,R_*)\rightarrow (L_*^\dagger,R_*^\dagger)$ is a conjugate-linear, period
two algebra homomorphism from $M_{cb}(L^1(\G))$ to $M_{cb}(L^1(\G))$.
This map extends the map $L^1(\hat\G) \rightarrow L^1(\hat\G); \hat\omega \mapsto
\hat\omega^*$.  The following is easy to deduce from Theorem~\ref{thm:one}.

\begin{proposition}\label{prop:twelve}
The homomorphism $\hat\Lambda:M_{cb}(L^1(\hat\G))\rightarrow C^b(\G)$ maps into
$D(S^{-1}) = D(S)^*$.  Furthermore, for $(L_*,R_*) \in M_{cb}(L^1(\hat\G))$, we have
that $\hat\Lambda(L_*^\dagger,R_*^\dagger) = S(\hat\Lambda(L_*,R_*)^*)$.
\end{proposition}

Informally, this means that we can ``see'' the (unbounded) antipode at the level
of two-sided multipliers.  From the remarks after Theorem~\ref{thm:one} that
the image of $\hat\lambda$, and hence certainly 
the image of $\hat\Lambda$, is a strict core for $S$ (as an operator on $C^b(\G)$).
We remark that in the classical case, when $\G=G$ a locally compact group, then
$S$ is bounded, but $M_{cb}A(G)$ need not be norm dense in $C^b(G)$ (but it
is of course always strictly dense).

To finish, we make links with Section~\ref{sec:links_with_uni}, and show how our
consideration of $\mc L(A,A\otimes K)$ as a ``slice'' of $\mc L(A\otimes K)$
is more than a technical tool.

\begin{theorem}\label{thm:four}
Let $(\alpha,\beta)$ be an invariant pair in $\mc L(C_0(\G),C_0(\G)\otimes K)$.
There exists a contraction $\mc T\in\mc L(C_0(\G)\otimes K)$ and $\xi_0,\eta_0\in K$
with $\|\xi_0\|=\|\alpha\|$ and $\|\eta_0\|=\|\beta\|$ such that
$\alpha = \mc T(\iota\otimes\xi_0)^*$ and $\beta = \mc T(\iota\otimes\eta_0)^*$.
\end{theorem}
\begin{proof}
We shall suppose, by rescaling, that $\|\alpha\|=\|\beta\|\leq1$.
We first show that $\beta^*\alpha = \epsilon 1$ for some $\epsilon\in\mathbb C$
with $|\epsilon|\leq 1$.  Indeed, let $L_*\in\mc{CB}(L^1(\hat\G))$ be the
left multiplier induced by $(\alpha,\beta)$.  Then $\beta^*\alpha = \tilde\beta^*
\tilde\alpha = \tilde\beta^*(1\otimes 1)\tilde\alpha = L(1)$.  Now, for $\hat\omega,
\hat\sigma\in L^1(\hat\G)$, we have that
\[ \ip{\Delta(L(1))}{\hat\omega\otimes\hat\sigma}
= \ip{1}{L_*(\hat\omega \hat\sigma)} = \ip{1}{L_*(\hat\omega) \hat\sigma}
= \ip{\Delta(1)}{L_*(\hat\omega) \otimes \hat\sigma}
= \ip{L(1)\otimes 1}{\hat\omega\otimes\hat\sigma}. \]
Thus $\Delta(L(1)) = L(1)\otimes 1$.  It follows from (the von Neumann version of)
\cite[Result~5.13]{kus} (see also \cite[Lemma~4.6]{aristov}) that $L(1) \in
\mathbb C1$, as required.  As $\|\beta^*\alpha\|\leq1$, it follows that
$|\epsilon|\leq1$.

Suppose for now that $|\epsilon|<1$.  Let $\xi_0$ and $\xi_1$ be orthogonal
unit vectors in $K$.  Choose $\delta$ with $|\epsilon|^2 + |\delta|^2=1$; by our
assumption, $\delta\not=0$.  Set $\eta_0 = \overline\epsilon\xi_0 + \delta\xi_1$,
and define
\[ \mc T = \alpha(\iota\otimes\xi_0)
+ \delta^{-1} (\beta-\overline\epsilon\alpha) (\iota\otimes\xi_1). \]
Then $\mc T(\iota\otimes\xi_0)^* = \alpha$ and $\mc T(\iota\otimes\eta_0)^*
= \overline\epsilon\alpha + \delta\delta^{-1} (\beta-\overline\epsilon\alpha)
= \beta$, as required.
It remains to show that $\mc T$ is a contraction.  It suffices to show that
$\|\mc T(\tau)\|\leq\|\tau\|$ for all $\tau\in A\otimes K$ of the form
$T=a\otimes \xi_0 + b\otimes \xi_1$, for some $a,b\in C_0(\G)$.  Indeed,
as the span of $\xi_0$ and $\xi_1$ agrees with the span of $\xi_0$ and
$\eta_0$, we may suppose that $\tau = a\otimes\xi_0 + b\otimes\eta_0$.
Then $\mc T(\tau) = \alpha(a) + \beta(b)$, so
\begin{align*} \|\mc T(\tau)\|^2 &= (\alpha^*\alpha(a)|a) + (\beta^*\alpha(a)|b)
+ (b|\beta^*\alpha(a)) + (\beta^*\beta(b)|b) \\
&\leq \|a\|^2 + \overline\epsilon(a|b) + \epsilon(b|a) + \|b\|^2 \\
&= (a|a) + (\xi_0|\eta_0)(a|b) + (\eta_0|\xi_0)(b|a) + (b|b) \\
&= \big( a\otimes\xi_0 + b\otimes\eta_0 \big| a\otimes\xi_0 + b\otimes\eta_0 \big)
= \|\tau\|^2.
\end{align*}
Thus $\mc T$ is a contraction.

If $|\epsilon|=1$, then $\alpha$ must be an isometry, for if
$\|\alpha(a)\|<\|a\|$ for some $a\in C_0(\G)$, then $\|a\|>\|\beta^*\alpha(a)\|
= |\epsilon|\|a\|$, a contradiction.  Similarly, $\beta$ is an isometry.  It follows
that $(\alpha-\epsilon\beta)^*(\alpha-\epsilon\beta)=0$, showing that
$\alpha=\epsilon\beta$.  Hence in this case, we can simply set
$\eta_0 = \overline\epsilon \xi_0$ and $\mc T=\alpha(\iota\otimes\xi_0)$.
\end{proof}



If $\alpha=\mc T(\iota\otimes\xi_0)^*$ and $\beta = \mc T(\iota\otimes\eta_0)^*$,
then the proof of Proposition~\ref{prop:five} shows that
\[ (1\otimes\beta)^*(\Delta*\alpha) = (\iota\otimes\iota\otimes\omega_{\eta_0,\xi_0})
T_{23}^*W_{12}^* T_{23} W_{12}. \]
Hence invariance can be expressed directly at the level of $T$; this of course is
taking us very far from our analogies with $M_{cb}A(G)$ and Gilbert's result.
Let us finish by looking at two-sided multipliers.

\begin{theorem}
Let $(L_*,R_*)$ be a completely bounded two-sided multiplier of $L^1(\hat\G)$.
There exists a Hilbert space $K$ with an involution $J_K$,
$\mc T\in\mc L(C_0(\G)\otimes K)$, and $\xi_0,\eta_0\in K$ such that:
\begin{enumerate}
\item with $\alpha=\mc T(\iota\otimes\xi_0)^*$ and $\beta=\mc T(\iota\otimes\eta_0)^*$,
  we have that $(\alpha,\beta)$ is invariant, and induces $L_*$;
\item with $\gamma=\mc T(\iota\otimes J_K\eta_0)^*$ and
  $\delta=\mc T(\iota\otimes J_K\xi_0)^*$, we have that $(\gamma,\delta)$ is invariant,
  and induces $\hat\kappa_*R_*\hat\kappa_*$ (and thus, using Section~\ref{ontheright},
  induces $R_*$).
\end{enumerate}
\end{theorem}
\begin{proof}
By rescaling, suppose that $\|(L,R)\|_{cb}=1$, so that $\|L\|_{cb}\leq 1$ and
$\|R\|_{cb}\leq 1$.  Apply the previous theorem to an invariant pair which induces
$L_*$ to form $T_1\in\mc L(C_0(\G)\otimes K_1)$, say,
with $\xi_0^{(1)}, \eta_0^{(1)}\in K_1$.  Similarly, find
$T_2\in\mc L(C_0(\G)\otimes K_2)$ and $\xi_0^{(2)}, \eta_0^{(2)}\in K_2$ for
$\hat\kappa_*R_*\hat\kappa_*$.  Indeed, looking at the proof of Theorem~\ref{thm:four},
we have that $\xi_0^{(1)}$ and $\xi_1^{(1)}$ are orthogonal unit vectors,
and that $\eta_0^{(1)} = \overline\epsilon_1\xi_0^{(1)} + \gamma_1\xi_1^{(1)}$,
where $L(1) = \epsilon_1 1$.  We have a similar construction for $\hat\kappa_*R_*
\hat\kappa_*$; in particular, $\epsilon_2 1 = \hat\kappa R \hat\kappa(1) = R(1)$.
Now, that $(L_*,R_*)$ is a two-sided multiplier means that $\hat\omega L_*(\hat\sigma)
= R_*(\hat\omega)\hat\sigma$ for $\hat\omega,\hat\sigma\in L^1(\hat\G)$.
Equivalently, $(\iota\otimes L)\hat\Delta = (R\otimes\iota)\hat\Delta$, and so
\[ \epsilon_1 1\otimes 1 = 1 \otimes L(1) = (\iota\otimes L)\hat\Delta(1)
= (R\otimes\iota)\hat\Delta(1) = R(1)\otimes 1 = \epsilon_2 1\otimes 1, \]
showing that $\epsilon_1=\epsilon_2$.  Remember that we have a free choice for
$\gamma_1$ and $\gamma_2$, subject to the condition that $|\gamma_1|^2 =
1 - |\epsilon_1|^2 = 1 - |\epsilon_2|^2 = |\gamma_2|^2$.  We shall assume that
$\gamma_1 = \overline{\gamma_2}$.

Let $\{ \xi_0^{(1)}, \xi_1^{(1)} \} \cup \{ e_i \}$ be an orthonormal basis for
$K_1$, and let $\{ \xi_0^{(2)}, \xi_1^{(2)} \} \cup \{ f_i \}$ be an orthonormal basis
for $K_2$.  By embedding $K_1$ or $K_2$ in a larger Hilbert space, if necessary, we
may suppose that $\{e_i\}$ and $\{f_i\}$ are indexed by the same set.
Let $K = K_1 \oplus K_2$, and let $J_K$ be the unique involution on $K$ which
satisfies
\[ J_K\big( \xi_0^{(1)} \big) = \eta_0^{(2)}, \quad
J_K\big( \xi_1^{(1)} \big) = \gamma_1\xi_0^{(2)} - \epsilon_1\xi_1^{(2)},
\quad J_K(e_i) = f_i. \]
For this to make sense, we need that for all $a,b,c,d\in\mathbb C$, we have
\begin{align*} \overline{a} c + \overline{b} d
&= \big( a\xi_0^{(1)} + b\xi_1^{(1)} \big| c\xi_0^{(1)} + d\xi_1^{(1)} \big)
= \big( J_K(c\xi_0^{(1)} + d\xi_1^{(1)}) \big| J_K(a\xi_0^{(1)} + b\xi_1^{(1)}) \big) \\
&= \big( \overline{c}\overline{\epsilon_1}\xi^{(2)}_0 + \overline{c}\gamma_2\xi^{(2)}_1
   + \overline{d}\gamma_1\xi^{(2)}_0 - \overline{d}e^{it}\epsilon_1\xi^{(2)}_1 \big|
   \overline{a}\overline{\epsilon_1}\xi^{(2)}_0 + \overline{a}\gamma_2\xi^{(2)}_1
   + \overline{b}\gamma_1\xi^{(2)}_0 - \overline{b}e^{it}\epsilon_1\xi^{(2)}_1 \big) \\
&= (c\epsilon_1+d\overline{\gamma_1})
   (\overline{a}\overline{\epsilon_1}+\overline{b}\gamma_1)
   + (c\overline{\gamma_2}-de^{-it}\overline{\epsilon_1})
   (\overline{a}\gamma_2-\overline{b}e^{it}\epsilon_1) \\
&= \overline{a}c (|\epsilon_1|^2 + |\gamma_2|^2)
   + \overline{b}d(|\gamma_1|^2+|\epsilon_1|^2)
   + c\overline{b}(\gamma_1 - \overline{\gamma_2})\epsilon_1
   + \overline{a}d(\overline{\gamma_1} - \gamma_2)\overline{\epsilon_1}.
\end{align*}
This holds, as $\gamma_1 = \overline{\gamma_2}$, $\epsilon_1=\epsilon_2$, and
$|\epsilon_1|^2 + |\gamma_1|^2=1$.  Notice that
\begin{align*} J_K\big( \eta_0^{(1)} \big) &=
\epsilon_1 J_K\big( \xi_0^{(1)} \big) + \overline{\gamma_1} J_K\big( \xi_1^{(1)} \big)
= \epsilon_1 \eta^{(2)}_0 + \overline{\gamma_1}\big( \gamma_1 \xi^{(2)}_0
   - \epsilon_1\xi^{(2)}_1 \big) \\
&= \epsilon_1 \overline{\epsilon_1} \xi^{(2)}_0
   + \epsilon_1\overline{\gamma_1}\xi^{(2)}_0
   + \overline{\gamma_1}\big( \gamma_1 \xi^{(2)}_0 - \epsilon_1\xi^{(2)}_1 \big)
= \xi^{(2)}_0.
\end{align*}

We have that $C_0(\G)\otimes K = C_0(\G)\otimes K_1 \oplus C_0(\G)\otimes K_2$
for the obvious isomorphism.  Let
\[ \mc T = \begin{pmatrix} \mc T_1 & 0 \\ 0 & \mc T_2 \end{pmatrix}
\in \mc L(C_0(\G)\otimes K). \]
Then, with $\alpha = \mc T(\iota\otimes \xi^{(1)}_0)^*
= \mc T_1(\iota\otimes\xi^{(1)}_0)^*$ and $\beta = \mc T(\iota\otimes \eta^{(1)}_0)^*
= \mc T_1(\iota\otimes\xi^{(1)}_0)^*$, we have that $(\alpha,\beta)$ induces $L_*$.
Also, with
\[ \gamma = \mc T(\iota\otimes J_K\eta^{(1)}_0)^* =
\mc T(\iota\otimes \xi^{(2)}_0)^*, \quad
\delta = \mc T(\iota\otimes J_K\xi^{(1)}_0)^* =
\mc T(\iota\otimes \eta^{(2)}_0)^*, \]
we have that $(\gamma,\delta)$ induces $\hat\kappa_*R_*\hat\kappa_*$, as we hoped.
\end{proof}

While the formulas in the above theorem are nicely symmetric, the proof feels
a little like a ``trick'' (although it is far from being completely artificial, as
we do use that $L_*$ and $R_*$ interact as a two-sided multiplier).
It is still our belief that there should be a more elegant approach to two-sided
multipliers.

In particular, let us finish with a question.  Let $(\alpha,\beta)$ be an
invariant pair, leading to a left multiplier $L$.  Can we ``see'', at the level
of the maps $\alpha$ and $\beta$, when there is a right multiplier $R$ making
the pair $(L,R)$ a two-sided multiplier?

\bigskip
\noindent\textbf{Author's address:}
\parbox[t]{5in}{School of Mathematics,\\
University of Leeds,\\
Leeds LS2 9JT\\
United Kingdom}

\smallskip
\noindent\textbf{Email:} \texttt{matt.daws@cantab.net}

\end{document}